\documentclass[journal,twoside,web]{ieeecolor}
\let\labelindent\relax

\usepackage{generic}
\usepackage{cite}
\usepackage[dvipsnames]{xcolor}

\usepackage{easyReview}

\usepackage{amsmath,amssymb,amsthm,graphics,epsfig,times,enumitem,float,booktabs}
\usepackage{hyperref,stmaryrd,comment,filecontents,bm} 
\hypersetup{colorlinks=true, linkcolor=blue, filecolor=blue, urlcolor=blue,citecolor=blue}
\usepackage[linesnumbered,ruled]{algorithm2e}

\usepackage[export]{adjustbox} 

\usepackage{caption,subcaption}
\usepackage{times,mathptm}
\usepackage{tikz,tikz-3dplot}
\usetikzlibrary{arrows,shapes,calc}
\usetikzlibrary{calc,intersections,through,backgrounds}
\usetikzlibrary{positioning}
\usepackage{pgfplots}
\pgfplotsset{compat=1.16}
\usepackage{pgfplotstable}
\usepgfplotslibrary{units}
\theoremstyle{theorem}
\newtheorem{theorem}{Theorem}
\newtheorem{corollary}{Corollary}
\newtheorem{lemma}{Lemma}
\theoremstyle{definition}
\newtheorem{definition}{Definition}

\newtheorem*{problem*}{Problem Statement}
\newtheorem{assumption}{Assumption}
\theoremstyle{remark}
\newtheorem{remark}{Remark}
\newtheorem{example}{Example} 
\DeclareMathOperator*{\argmax}{argmax}

\newcommand{\cufrac}[2]{\tfrac{\partial #1}{\partial #2}} 
\usepackage{textcomp}
\def\BibTeX{{\rm B\kern-.05em{\sc i\kern-.025em b}\kern-.08em
		T\kern-.1667em\lower.7ex\hbox{E}\kern-.125emX}}
\def\B{{\mathbb B}}

\def\x{\mathbf{x}}

\def\B{\mathit{B}}

\def\mR{\mathsf R}
\def\R{{\mathbb R}}
\def\u{\mathbf u}
\def\v{\mathbf v}
\def\H{\mathit H}
\def\C{\mathcal C}
\def\T{\mathsf T}

\def\V{\mathcal V}
\def\G{{\bm\gamma^{\star}}}
\def\Th{{\bm\theta^{\star}}}

\begin{document}
	\title{Optimal Role Assignment for Multiplayer Reach-Avoid Differential Games in 3D Space}
	\author{Abinash Agasti, Puduru Viswanadha Reddy, \IEEEmembership{Member, IEEE}, Bharath Bhikkaji
		\thanks{A. Agasti, P. V. Reddy and B. Bhikkaji are with the Department of Electrical Engineering, Indian Institute of Technology-Madras, Chennai, 600036, India.  
			(e-mail: abinashagasti@outlook.com, vishwa@ee.iitm.ac.in, bharath@ee.iitm.ac.in)}} 
	
	\maketitle
	\thispagestyle{empty}
	\begin{abstract} 
		In this article an $n$-pursuer versus $m$-evader reach-avoid differential game in 3D space is studied. A team of evaders aim to \textit{reach} a stationary target while \textit{avoiding} capture by a team of pursuers. The multiplayer scenario is formulated in a differential game framework. This article provides an optimal solution for the particular case of $n=m=1$ and extends it to a more general scenario of $n\geq m$ via an optimal role assignment algorithm based on a linear program. Consequently, the pursuer and the evader winning regions, and the Value of the game are analytically characterized providing optimal strategies of the players in state feedback form.
		
	\end{abstract}
	
	
	\section{Introduction}   
	Multiplayer reach-avoid games  are mathematical abstractions of strategic interactions involving a team of \textit{evaders} that strives to reach a predetermined goal while avoiding an adversarial team of \textit{pursuers}. Such interactions arise in various real-world scenarios, including safe motion/path planning, region protection in the presence of hostile agents, dynamic collision avoidance and robotic herding \cite{tomlin_2015,karaman_frazzoli_2011,oyler_2016,paranjape_2018,chung_2011}. 
	
	
	Differential game theory \cite{isaacs_1965,basar_1999} is a commonly used mathematical framework to model and analyze multiplayer reach-avoid games. The general approach in any such reach-avoid differential game (RADG) first involves finding the winning regions of both the parties, referred to as Game of Kind (GoK). After solving the GoK, the aim is then to find the optimal strategies of the players in state feedback form, referred to as Game of Degree (GoD). This requires the solution to the Hamilton-Jacobi-Isaacs partial differential equation (HJI-PDE) for the reach-avoid scenario. Using this differential game theory framework, a wide variety of two player reach-avoid games have been solved in the literature \cite{margellos_2011,yan_2d1a_2019,yan_guarding_2022,tomlin_2005}. However, solving the HJI-PDE analytically for a multiplayer scenario has been shown to be nontrivial and a challenging task \cite{garcia_2019,pachter_2018,pachter_2019,moll_2022}. 
	To solve a multiplayer reach-avoid differential game (MRADG), two player RADG solutions can be used as building blocks. This involves pairing pursuers to evaders, thus breaking down the multiplayer game into multiple two player subgames. This multilayered solution approach requires an optimal assignment of pursuers to evaders (involving discrete variables), and optimal strategies for the consequent interactions (involving continuous variables). The solutions to the smaller subgames must then be aggregated to obtain a solution of the multiplayer HJI-PDE to provide a guarantee of optimality. Thus, the MRADG is posed as a hybrid multi-agent decision problem. 
	
	
	
	In the recent years, there have been multiple attempts to address the hybrid nature of the problem arising in MRADGs in 2D space. In \cite{tomlin_2017,cdc_23}, the authors first design a computationally efficient solution to a small subgame and then use a maximum matching strategy for assignment of multiple players. Zhou  et al.  \cite{tomlin_2018} aim to reduce the computational burden through approximated solutions of the HJI-PDE. In \cite{yan_assignment_2020}, the authors analytically solve the GoK but do not provide optimal strategies (GoD).  A complete analytical solution to a multiplayer game with homogenous agents was provided along with the optimal strategies in state feedback form by Garcia et al. \cite{garcia_2021}. However, to determine the optimal assignment, they enumerate all feasible assignments relevant to the game, resulting in an exponential increase in complexity with the number of players.
	
	Motivated by military and surveillance applications that require aerial swarms to operate in adversarial environments \cite{Chung:18}, MRADGs with players interacting in 3D space have been considered in the literature.  In \cite{garcia_2020}, the authors provide solutions to both GoK and GoD for an RADG involving $2$ pursuers and $1$ evader with all agents moving at the same speed. In another work, Yan et al. \cite{yan_matching_2022}, the authors present an analytical solution for a subgame with multiple pursuers and a single evader. This solution is then used as a building block to develop a polynomial-time constant factor approximation for an NP-hard assignment algorithm. As it is an approximate solution, it does not satisfy the multiplayer HJI-PDE. In summary, all existing approaches for solving an MRADG (in 2D or 3D space) extend the solution obtained from smaller subgames. In the proposed approaches found in the literature so far, either a computationally feasible extension from a smaller subgame or the obtaining of optimal strategies in closed-form (by solving the HJI-PDE) has been considered. However, a computationally feasible extension simultaneously obtaining the optimal strategies in closed-form has not been explored in the literature.  
	
	\subsubsection*{Contributions} 
	The contribution of this paper lies in providing computationally efficient solutions to  a class of MRADGs involving $m$ evaders and $n$ pursuers interacting in 3D space. In particular, we introduce a cost-benefit framework for optimal role assignment of pursuers to evaders. Using this, we obtain
    optimal strategies of the players (in closed-form) by solving the corresponding HJI-PDE. The main results of our article are summarized as follows:
	\begin{enumerate}
		\item In Section \ref{sec:1v1}, we study the \texttt{1v1} subgame involving one pursuer and one evader with unequal speeds. First, we begin by characterizing the pursuer and evader winning regions (GoK) in Lemma \ref{lem:1v1barrier}. Then, in Theorem \ref{theorem:1v1_varspeed} and Theorem \ref{theorem:1v1_varspeed_E}, we obtain the Value function of the \texttt{1v1} game in the pursuer and evader winning region respectively (GoD). This facilitates the provision of optimal strategies of the players in state feedback form. We note that these results extend the work in \cite{garcia_2020}, where players were assumed to have equal speeds in a similar scenario.
		\item  Under few assumptions on players' interactions (Assumption \ref{assum:assumption1}) in an MRADG, we provide a cost-benefit framework for solving the optimal assignment of pursuers to evaders as a Shapley-Shubik assignment game \cite{shapley_1971}. In particular, in Theorem \ref{thm:lp}, we show that the set of optimal assignments can be obtained by solving a linear programming problem. In Lemma \ref{lemma:win} and Lemma \ref{lemma:invariance}, we show that the consequence of the game under optimal play is invariant with these optimal assignments. Finally, in Theorem \ref{thm:multibarrier}, we obtain the winning regions of the pursuer and evader teams (GoK).			
		\item In Theorem \ref{thm:valueP}, we obtain the optimal strategies of the players in the pursuer winning region (GoD). Then, we demonstrate in Lemma \ref{lem:refine}, that the set of optimal assignments requires a refinement for the evader winning region, for which we obtain the optimal strategies of the players in Theorem \ref{thm:valueE} (GoD).
	\end{enumerate}
	\subsubsection*{Novelty and Differences with Existing Literature}
	The novel aspects of our paper and differences with the existing literature are  summarized as follows:
	
	\begin{enumerate}
		\item While prior literature often assumes agents with homogeneous capabilities or superior speed pursuers \cite{yan_guarding_2022,pachter_2019,moll_2022,garcia_2020,garcia_2021}, we introduce a more realistic scenario by allowing some evaders to be faster than some pursuers. 
		\item Typically in the related literature, 	first a solution satisfying the HJI-PDE for the \texttt{1v1} case is found, which is then extended to multiplayer scenarios using matching algorithms \cite{tomlin_2017,yan_assignment_2020,yan_matching_2022}.  However, this approach lacks optimality as it doesn't satisfy the HJI-PDE for the multiplayer game. In this work, an assignment scheme is designed that aids in finding a solution that satisfies the HJI-PDE for the multiplayer game. 
		\item The assignment scheme introduced in this work is fundamentally a maximum weighted matching algorithm while most of the literature in multiplayer setting \cite{tomlin_2017,yan_assignment_2020,garcia_2021,yan_matching_2022} uses a version of the maximum matching algorithm. While being able to ensure capturing the most number of evaders, the maximum matching algorithms fail to distinguish between a pair of assignments leading to the capture of the same number of evaders. But, the assignment scheme in our work distinguishes between any such pair of assignments based on the payoff received from the underlying interactions dictated by the matchings. 
		\item 
		Our work provides a comprehensive solution for the MRADG by addressing GoDs in both the pursuer and evader winning regions as opposed to existing literature \cite{pachter_2018,yan_guarding_2022,garcia_2020,garcia_2021} where solution is found primarily in the pursuer winning region. 
	\end{enumerate}
	
	
	\subsubsection*{Organization}
	The article is organized as follows. The problem formulation of the MRADG is stated in Section \ref{sec:prelim}. Section \ref{sec:1v1} studies the \texttt{1v1} subgame involving one pursuer and one evader. Using these results, in Section \ref{sec:multiplayer} we provide a cost-benefit framework to design an optimal assignment of pursuers to evaders. Then, the assignment is used to characterize the pursuer and evader winning regions in Section \ref{subsec:gok} and provide the Value function of the multiplayer game in Section \ref{subsec:god}. Section \ref{sec:num} provides a few numerical illustrations, and finally conclusions are drawn in Section \ref{sec:conc} along with citing some future directions.

	\section{Preliminaries and Problem Formulation}
	\label{sec:prelim}
	We consider an MRADG consisting of $n$ pursuers and $m$ evaders. The evaders aim to reach a stationary target, while the pursuers aim to prevent this outcome. The target is assumed without loss of generality to be the origin. A player in the evading team is denoted by $E_i$, $i\in M:=\{1,\cdots,m\}$, and one in the pursuing team by $P_j$, $j\in N:=\{1,\cdots,n\}$. The players are assumed to be holonomic, and interact in  3-dimensional Euclidean space. The position vectors (or states) of $E_i$ and $P_j$ are denoted respectively by $\x_{E_i}:=(x_{E_i} ,y_{E_i} ,z_{E_i})\in \mathbb{R}^3$ and $\x_{P_j} :=(x_{P_j} ,y_{P_j} ,z_{P_j} )\in \mathbb{R}^3$  for   $i\in M$  and $j\in N$. The global state space of the differential game is denoted by $\x:=(\x_E,\x_P)\in \mathbb R^{3(m+n)}$, where $\x_E:=(\x_{E_1},...,\x_{E_m})$ and $\x_P:=(\x_{P_1},...,\x_{P_n})$.  We denote the controls of the players $E_i$ and $P_j$ respectively by $\u_{E_i}:=(u_{x_i},u_{y_i},u_{z_i})$  and $\v_{P_j}:=(v_{x_j},v_{y_j},v_{z_j})$  for $i\in M$ and $j\in N$. The joint team controls for the evaders and pursuers are denoted by $\u:=(\u_{E_1},...,\u_{E_m})$ and $\v:=(\v_{P_1},...,\v_{P_n})$ respectively. We assume that players $E_i$ and $P_j$ move with constant speeds $U_i> 0$ and $V_j> 0$ respectively for $i\in M$ and $j\in N$. Consequently, the speed ratios $\alpha_{ij}:=U_i/V_j$ can take any real positive value for all $i\in M$ and $j\in N$. Now, the admissible control sets of $E_i$ and $P_j$ are given respectively by $\mathcal U_i:=\{\u_{E_i}\in \mathbb R^3:||\u_{E_i}|| =U_i \}$ and $\mathcal V_i:=\{\v_{P_j}\in \mathbb R^3:||\v_{P_j}|| =V_j \}$, where $||.||$ denotes the Euclidean norm. The admissible controls for the evader and pursuer teams are then given by the cartesian products $\mathcal U:=\prod\limits_{i\in M}\mathcal U_i$ and $\mathcal V:=\prod\limits_{j\in N}\mathcal V_j$.
	
	The players have simple motion dynamics given by

	\begin{equation}
		\begin{aligned}
			&\dot{x}_{E_i}(t) = u_{x_i}(t),\quad \dot{y}_{E_i}(t) = u_{y_i}(t),\quad \dot{z}_{E_i}(t) = u_{z_i}(t), \\
			&\dot{x}_{P_j}(t) = v_{x_j}(t),\quad \dot{y}_{P_j}(t) = v_{y_j}(t),\quad \dot{z}_{P_j}(t) = v_{z_j}(t)
			\label{eq:multi_dynamics}
		\end{aligned} 
	\end{equation}
	with initial positions $\x_{E_i}(0) = (x_{E_{i_0}}, y_{E_{i_0}}, z_{E_{i_0}})$ and $\x_{P_j}(0) = (x_{P_{j_0}}, y_{P_{j_0}}, z_{P_{j_0}})$ for $i\in M$ and $j\in N$. The global initial state of the system is denoted by $\x_0 \in \mathbb{R}^{3(m+n)}$. 
	\begin{remark}
		Note that the agents have a single integrator dynamics that might seem simplistic to model any real-world autonomous system. However, many autonomous systems often have a dynamics that is affine in control or is feedback linearizable \cite{Francis_Maggiore_2016}, i.e., of a form $\dot\x=f(\x)+g(\x)\u$. In such cases, a simple transformation $\u=g^{-1}(\x)(-f(\x)+\tilde\u)$ converts the system into a single integrator model $\dot\x=\tilde\u$. Thus, the single integrator model can be used to provide analytical results on the optimal behaviour of complex multiplayer interactions. These solutions serve as a high-level controller which can be used to drive a low-level controller that further drives the actuators of the autonomous system. 
	\end{remark}
	
	\subsubsection*{Termination criterion and termination sets}
	The MRADG considered in this paper is characterized by two termination criteria: either all the evaders are captured by the pursuer team or at least one of the evaders reaches the origin. An evader is defined to have reached the origin when its state vector coincides with the origin. On the other hand, it is defined to be captured when its state vector coincides with that of a pursuer. Once an evader is captured, both the captured evader and the capturing pursuer come to a complete stop. In other words, their speeds drop to zero, and they cease to participate further in the game. 	To represent the terminal criteria, we first define binary variables $\{\mu_{ij}\in \{0,1\},~ i\in M,~j\in N\}$, where $\mu_{ij}=1$ when $P_j$ is assigned to capture $E_i$, and $0$ otherwise.
	
	The termination set for the game is a subset of the global state space defined by
	\begin{equation}
		\T:=\T_E\cup\T_P, \label{eq:termination}
	\end{equation}
	where 
	\begin{equation}
		\begin{aligned}
			\T_E&:=\{\x\in \mathbb R^{3(n+m)}~\big\vert~ \exists \tilde M\subset M \\
			&\text{ such that } \forall i\in\tilde M \text{ we have }\vert\vert\x_{E_i} \vert\vert=0
			\}
		\end{aligned} 
		\label{eq:terminalE}
	\end{equation}
	represents the game outcome when  at least one of the evaders is capable of reaching the target. Note that the set $\tilde M$ is assumed to be the maximal set, i.e., there do not exist any more evaders who manage to win against the pursuers assigned to them. Next, the game outcome when all the evaders are captured by the pursuer team away from the target is represented by
	\begin{equation}
		\begin{aligned}
			\T_P=&\{ \x\in \mathbb R^{3(n+m)}~\big|~\forall i\in M\ \exists j\in N \text{ such that } \\
			&\mu_{ij}=1 \text{ and }	\vert\vert \x_{P_j}-\x_{E_i} \vert\vert=0\text{ and } \vert\vert \x_{E_i} \vert\vert>0 \}.
		\end{aligned}
	\end{equation}
	Further, the associated termination time is given by
	\begin{equation}
		t_f:=\inf\{t\in \R_+: \x(t)\in\T\},
	\end{equation} where $\R_+$ denotes the nonnegative real line.  

	\subsubsection*{Game of Kind and Game of Degree}
	As there are two possible outcomes in any MRADG,  
	we solve the GoK to obtain a partition of the global state space into winning regions for the pursuer and evader teams. All the initial positions in the state space are classified into the pursuer winning region if there exists a strategy for the pursuer team that guarantees a win independent of the evader team strategies. Note that a win for the pursuer team is defined as all the evaders being captured away from the target. Similarly, all the initial positions in the state space are classified into the evader winning region if there exists a strategy for the evader team that guarantees a win independent of the pursuer team strategies. Note that a win for the evader team is defined as at least one evader reaching the target. Further, the initial positions that result in a pursuer and an evader reaching the target simultaneously render the target unsafe and hence are counted in the evader winning region.
	
	 Let $\mR_P$ and $\mR_E$ denote the pursuer and evader winning regions respectively which partition the global state space $\R^{3(m+n)}$. Then, we solve the GoD within these regions to obtain optimal strategies of the players in state feedback form. To this end, we first define the zero-sum payoff for the multiplayer game separately in $\mR_P$ and $\mR_E$. Let the terms $\x_{E_{i_f}}$ and $\x_{P_{j_f}}$ denote the terminal positions of $E_i$ and $P_j$ respectively. Now, consider the cost function for the game starting in the region $\mR_E$ defined as
	\begin{equation}
		J(\u(.),\v(.);\x_0)=-\sum\limits_{i\in\tilde{M}}\min\limits_{j\in \tilde{N}_i}\vert\vert\x_{P_{j_f}}\vert\vert. \label{eq:multi_Ecost}
	\end{equation}
	Here, $\tilde{N}_i=\{j\in N:\mu_{ij}=1\}$ is the set of all pursuers assigned to evader $E_i$, and $\tilde M$ is as defined in \eqref{eq:terminalE}. This cost accumulates the individual costs corresponding to each evader $E_i$ reaching the target at the terminal time. 
	This individual cost is given by the closest distance of all pursuers assigned to $E_i$, from the target. The inclusion of the negative sign keeps the convention that pursuers are considered the maximizing team for the cost in \eqref{eq:multi_Ecost}. This convention also applies to the cost considered when starting in the region $\mR_P$, given by
	
	\begin{equation}
		J(\u(.),\v(.);\x_0)=\sum\limits_{i\in M}\vert\vert\x_{E_{i_f}}\vert\vert. \label{eq:multi_Pcost}
	\end{equation}
	This cost function represents an accumulation of individual costs corresponding to the terminal distance of each evader from the target.
	The optimal payoff in this game, known as the Value of the game, is defined for each initial state $\x_0\in\R^{3(m+n)}$ as
	\begin{equation}
		\V(\x_0):=\min\limits_{\u(.)}\max\limits_{\v(.)}J(\u(.),\v(.);\x_0) \label{eq:valfun}
	\end{equation} 
	subject to \eqref{eq:multi_dynamics} and \eqref{eq:termination}, where $\u(.)$ and $\v(.)$ are the teams' state feedback strategies.

	The Value function \eqref{eq:valfun} can be obtained as the solution to the Hamilton-Jacobi-Isaacs partial differential equation (HJI-PDE) given by 
	\begin{equation}
		\min\limits_{\u\in \mathcal U}\max\limits_{\v\in \mathcal V}\left[ \left\langle\nabla\V(\x),(\u,\v) \right\rangle \right] =0, \label{eq:hji}
	\end{equation}
	where $\mathcal U$ and $\mathcal V$ are the admissible control sets of the evader and pursuer teams respectively.
	\subsubsection*{Problem formulation and solution approach}	
	As mentioned in the introduction, an MRADG is a hybrid multiplayer decision problem involving discrete decision variables to obtaining optimal assignments of pursuers to evaders, as well as continuous decision variables to obtaining optimal strategies of the players. Naturally, the following questions arise while solving this co-design problem.
	\begin{enumerate}
		\item Is there a criterion that could be used for optimally assigning the pursuers to the evaders?
		\item Is there a way to characterize winning regions for the pursuer and evader teams, and obtain optimal strategies for the players in these regions? 
	\end{enumerate}
	In this paper, we address the above questions using a two-layered approach. First, we analyze  a smaller sub-game involving individual strategic interactions between a pursuer and an evader. This study enables us to compute optimal payoffs that players can obtain from the interaction. 
	In particular, we characterize  pursuer and evader winning regions (GoK) and analytically solve the corresponding HJI-PDE to obtain optimal strategies (GoD) of the players in state feedback form.
	Using the solutions from the smaller subgames, we then address the first question by proposing a cost-benefit framework for optimally assigning pursuers to evaders. We address the second question by constructing a barrier function induced by the optimal assignment, which partitions the state space into winning regions for the pursuer and evader teams. Further, we obtain state feedback optimal strategies for all the players in these regions. Before we proceed with our solution approach, we recall the following preliminary result which will be used throughout the remainder of the paper.
	\begin{lemma}
		Consider the MRADG described by  \eqref{eq:multi_dynamics},    \eqref{eq:multi_Ecost} and \eqref{eq:multi_Pcost}. The optimal headings are constant and the optimal trajectories are straight lines. \label{theorem:prelim}
	\end{lemma}
	\begin{proof}
		As the cost functions  \eqref{eq:multi_Ecost} and \eqref{eq:multi_Pcost} are of terminal type, the  Hamiltonian associated with the saddle-point problem in \eqref{eq:valfun}
		is written as 
		\begin{equation}
			\begin{aligned}
				\H &=\sum\limits_{i\in M}\langle \lambda_{E_i},\dot\x_{E_i}\rangle+\sum\limits_{j\in N} \langle\lambda_{P_j},\dot\x_{P_j}\rangle =\sum\limits_{i\in M} \langle \lambda_{E_i},\u_i\rangle+\sum\limits_{j\in N} \langle \lambda_{P_j},\v_j\rangle,
			\end{aligned}
		\end{equation}
		where  $\lambda=(\lambda_{E_1},...,\lambda_{E_m},\lambda_{P_1},...,\lambda_{P_n})\in\R^{3(m+n)}$ denotes the costate vector. As the Hamiltonian is independent of $\x$   the costate dynamics is obtained as $\dot\lambda=-\frac{\partial}{\partial \x}\H=0$. So, the costate vector remains constant under optimal play. The saddle-point controls must satisfy $\min_{\u}\max_{\v}H=0$ and thus depend on the costates. As the costates are constant, this implies that the optimal controls also remain constant.  Consequently, the optimal trajectories are straight lines. 
	\end{proof}
	%
	\section{1 versus 1 Differential Game}
	\label{sec:1v1}
	In this section, we analyze a reach-avoid differential game with one pursuer and one evader, referred to as the \texttt{1v1} RADG for brevity. To maintain the notation consistent with MRADG, we label the evader and the pursuer as $E_i$ and $P_j$ respectively. 
	
	The solution to the \texttt{1v1} RADG entails finding the Apollonius sphere, which is the locus of all the points that the pursuer and the evader can reach simultaneously. This desired locus is given by
	\[\frac{\vert\vert \x-\x_{P_j}\vert\vert^2}{V_j^2}=\frac{\vert\vert \x-\x_{E_i}\vert\vert^2}{U_i^2} \Rightarrow \alpha_{ij}^2\vert\vert \x-\x_{P_j}\vert\vert^2=\vert\vert \x-\x_{E_i}\vert\vert^2.\]
	Upon simplification, the Apollonius sphere is calculated as
	\begin{equation}
		(x-x_{c_{ij}})^2+(y-y_{c_{ij}})^2+(z-z_{c_{ij}})^2=r_{c_{ij}}^2, \label{eq:apsphere_E}
	\end{equation}
	where $x_{c_{ij}}=\frac{x_{E_i}-\alpha_{ij}^2x_{P_j}}{1-\alpha_{ij}^2}$, $y_{c_{ij}}=\frac{y_{E_i}-\alpha_{ij}^2y_{P_j}}{1-\alpha_{ij}^2}$, $z_{c_{ij}}=\frac{z_{E_i}-\alpha_{ij}^2z_{P_j}}{1-\alpha_{ij}^2}$, and $r_{c_{ij}}=\frac{\alpha_{ij}}{1-\alpha_{ij}^2}||\x_{P_j}-\x_{E_i}||$. 
	
	The following result pertains to the solution of the GoK for the \texttt{1v1} RADG.

	\begin{lemma}
		\label{lem:1v1barrier}
		The   function $\B_{ij}:\R^6\shortrightarrow\R$ defined by 
		\begin{equation}
			\B_{ij}(\x)=R_{E_i}^2-\alpha_{ij}^2R_{P_j}^2, \label{eq:1v1_barrier}
		\end{equation}
		referred to as the barrier function for the \texttt{1v1} RADG where $R_{E_i}=\vert\vert \x_{E_i}\vert\vert$ and $R_{P_j}=\vert\vert \x_{P_j}\vert\vert$. This barrier function partitions the state space into the pursuer winning region as $\mR_P:=\{\x\in\R^6:\B_{ij}(\x)>0\}$ and the evader winning region as $\mR_E:=\{\x\in\R^6:\B_{ij}(\x)\leq0\}$. 
	\end{lemma}
	\begin{proof}
		Since $\B_{ij}$ is well defined over its domain, the sets $\{\x\in\R^6:\B_{ij}(\x)>0\}$ and $\{\x\in\R^6:\B_{ij}(\x)\leq0\}$ form a partition of the state space $\R^6$. Note that the regions $\mR_P$ and $\mR_E$ also form a partition of the state space $\R^6$.
		
		
		Consider some $\x\in\R^6$ such that $\B_{ij}(\x)>0$, then $R_{E_i}^2>\alpha_{ij}^2R_{P_j}^2 \Rightarrow R_{E_i}^2/U_i^2>{R_{P_j}^2}/{V_j^2}$. Thus, starting from any point $\x\in\{\x\in\R^6:\B_{ij}(\x)>0\}$, the pursuer can reach the origin faster than the evader. This clearly corresponds to the pursuer winning, i.e., $\{\x\in\R^6:\B_{ij}(\x)>0\}\subseteq \mR_P$. Now, consider some $\x\in\R^{6}$ such that $\B_{ij}(\x)\leq0$, then $R_{E_i}^2/U_i^2\leq{R_{P_j}^2}/{V_j^2}$. Thus, starting from any point $\x\in\{\x\in\R^6:\B_{ij}(\x)\leq0\}$, the evader can reach the origin faster than the pursuer. This clearly corresponds to the evader winning, i.e., $\{\x\in\R^6:\B_{ij}(\x)\leq0\}\subseteq\mR_E$.
		
		Finally, consider the opposite direction, and let $\x\in\mR_P$. Suppose for contradiction $\B_{ij}(\x)\leq0$, then $\x\in\{\x\in\R^6:\B_{ij}(\x)\leq0\}$ which implies $\x\in\mR_E$. Hence $\B_{ij}(\x)$ must be positive resulting in $\mR_P\subseteq\{\x\in\R^6:\B_{ij}(\x)>0\}$. Using earlier part of the proof, this shows that $\mR_P=\{\x\in\R^6:\B_{ij}(\x)>0\}$. A similar result can be shown for the evader winning region to obtain $\mR_E=\{\x\in\R^6:\B_{ij}(\x)\leq0\}$.
		
		
	\end{proof}
	Now, we derive the optimal strategies of the players separately in the pursuer and evader winning
	regions in the next two results.  
	\begin{theorem}[Pursuer winning region]
		Consider the \texttt{1v1} RADG with $\x\in\mR_P$ and $\alpha_{ij}<1$. The Value function is $\C^1$ and it is the solution of the HJI-PDE \eqref{eq:hji}. The Value function is given by 
		\begin{equation}
			\V_{ij}^P(\x)=R_{c_{ij}}-r_{c_{ij}}, \label{eq:value_Pout}
		\end{equation}
		where $R_{c_{ij}}=\vert\vert \x_{c_{ij}} \vert\vert$, with $\x_{c_{ij}}=(x_{c_{ij}},y_{c_{ij}},z_{c_{ij}})$ 
		and $r_{c_{ij}}$ given by \eqref{eq:apsphere_E}. \label{theorem:1v1_varspeed}
	\end{theorem}
	\begin{proof}
		Since $\alpha_{ij}<1$, the evader is located inside the Apollonius sphere, and since  $\x\in\mR_P$, the target must lie outside the sphere. As both the players can reach any point on the Apollonius sphere in the same time, the evader must choose to head to the point on the sphere that minimizes its distance from the origin (target), considering the cost specified by \eqref{eq:multi_Pcost}. The coordinates of this point are determined as the solution of the following  equality constrained optimization problem  
		\begin{equation}
			\begin{aligned}
				\min\ &x^2+y^2+z^2,\\ \label{eq:opt_sphere}
				\text{subject to } (x-x_{c_{ij}}&)^2+(y-y_{c_{ij}})^2+(z-z_{c_{ij}})^2=r_{c_{ij}}^2.
			\end{aligned}
		\end{equation}
		Taking the Lagrange multiplier associated with the equality constraint as $\delta$, the first order necessary condition is obtained as  $
		x+\delta(x-x_{c_{ij}})=0$, $
		y+\delta(y-y_{c_{ij}})=0$, and $ 
		z+\delta(z-z_{c_{ij}})=0$.
		The desired optimal point is   obtained as $I=(x^*,y^*,z^*)=\frac{\delta}{1+\delta}(x_{c_{ij}},y_{c_{ij}},z_{c_{ij}})$. Since $I$ must lie on  the Apollonius sphere, it satisfies a quadratic equation given by $(x_{c_{ij}}^2+y_{c_{ij}}^2+z_{c_{ij}}^2)\left(1-\frac{\delta}{1+\delta}\right)^2=r_{c_{ij}}^2$. Solving for $\delta$, we get $1-\tfrac{\delta}{1+\delta}=1\mp\tfrac{r_{c_{ij}}}{R_{c_{ij}}}$. 
		One solution corresponds to the point on the sphere farthest from the origin, while the other solution provides the 
		interception point, given by:
		\begin{equation}
			I=(x^*,y^*,z^*)=\left( 1-\tfrac{r_{c_{ij}}}{R_{c_{ij}}}\right) (x_{c_{ij}},y_{c_{ij}},z_{c_{ij}}).
		\end{equation}
		The payoff of the game when the evader and the pursuer choose to head straight towards the interception point (as a result of Lemma \ref{theorem:prelim}) yields a guess of the Value function as $\V_{ij}^P(\x)=||I||=R_{c_{ij}}-r_{c_{ij}}$. 
		
		Next, we verify if the candidate Value function satisfies the HJI-PDE  \eqref{eq:hji}, written as
		\begin{multline}
			\min\limits_{\u_i}\max\limits_{\v_j}\left(\cufrac{\V_{ij}^P}{x_{E_i}} u_{x_i}+\cufrac{\V_{ij}^P}{y_{E_i}}u_{y_i}+\cufrac{\V_{ij}^P}{z_{E_i}}u_{z_i}\right. \\
			\left. +\cufrac{\V_{ij}^P}{x_{P_j}}v_{x_j}+\cufrac{\V_{ij}^P}{y_{P_j}}v_{y_j}+\cufrac{\V_{ij}^P}{z_{P_j}}v_{z_j}\right)=0. \label{eq:ME1}
		\end{multline}
		Then, the optimal controls are obtained in terms of the state and the Value function given by
		\begin{equation}
			\begin{aligned}
				\u_i^*&=-\tfrac{U_i}{\rho_{E_i}}\begin{bmatrix}
					\cufrac{\V_{ij}^P}{x_{E_i}}&\cufrac{\V_{ij}^P}{y_{E_i}}&\cufrac{\V_{ij}^P}{z_{E_i}}
				\end{bmatrix},\\  \v_j^*&=\tfrac{V_j}{\rho_{P_j}}\begin{bmatrix}
					\cufrac{\V_{ij}^P}{x_{P_j}}&\cufrac{\V_{ij}^P}{y_{P_j}}& \cufrac{\V_{ij}^P}{z_{P_j}}
				\end{bmatrix},\label{eq:opt_con}
			\end{aligned}
		\end{equation}
		where $\rho_{E_i}=\sqrt{\left( \cufrac{\V_{ij}^P}{x_{E_i}}\right) ^2+\left( \cufrac{\V_{ij}^P}{y_{E_i}}\right) ^2+\left( \cufrac{\V_{ij}^P}{z_{E_i}}\right) ^2}$ and $\rho_{P_j}=\sqrt{\left( \cufrac{\V_{ij}^P}{x_{P_j}}\right) ^2+\left( \cufrac{\V_{ij}^P}{y_{P_j}}\right) ^2+\left( \cufrac{\V_{ij}^P}{z_{P_j}}\right) ^2}$. Substituting the optimal controls \eqref{eq:opt_con} in the HJI equation \eqref{eq:ME1}, we get
		\begin{equation}
			-\alpha_{ij}\rho_{E_i}+\rho_{P_j}=0. \label{eq:ME2}
		\end{equation}
		To verify if the proposed Value function \eqref{eq:value_Pout}
		indeed satisfies \eqref{eq:ME2}, we first determine the partial derivatives of $\V_{ij}^P$ as
		\begin{equation}
			\begin{aligned}
				\begin{bmatrix}\cufrac{\V_{ij}^P}{x_{E_i}}&\cufrac{\V_{ij}^P}{y_{E_i}}&\cufrac{\V_{ij}^P}{z_{E_i}}\end{bmatrix}&=\tfrac{1}{1-\alpha_{ij}^2}\left(\tfrac{\x_{c_{ij}}}{R_{c_{ij}}} -\tfrac{\alpha_{ij}^2}{1-\alpha_{ij}^2}\tfrac{\x_{E_i}-\x_{P_j}}{r_{c_{ij}}}\right), \\
				\begin{bmatrix}\cufrac{\V_{ij}^P}{x_{P_j}}&\cufrac{\V_{ij}^P}{y_{P_j}}&\cufrac{\V_{ij}^P}{z_{P_j}}\end{bmatrix}&=\tfrac{\alpha_{ij}^2}{1-\alpha_{ij}^2}\left(-\tfrac{\x_{c_{ij}}}{R_{c_{ij}}} +\tfrac{1}{1-\alpha_{ij}^2}\tfrac{\x_{E_i}-\x_{P_j}}{r_{c_{ij}}}\right).
			\end{aligned}
			\label{eq:gradV_compact}
		\end{equation} 
		Using \eqref{eq:gradV_compact}, we get
		\begin{align*} 
		 \alpha_{ij}^2\rho_{E_i}^2&=\alpha_{ij}^2\left(\left( \cufrac{\V_{ij}^P}{x_{E_i}}\right) ^2+\left( \cufrac{\V_{ij}^P}{y_{E_i}}\right) ^2+\left( \cufrac{\V_{ij}^P}{z_{E_i}}\right) ^2\right)\\ &=\tfrac{\alpha_{ij}^2}{(1-\alpha_{ij}^2)^2}\left[1-\tfrac{2\alpha_{ij}^2}{1-\alpha_{ij}^2}\tfrac{\langle \x_{c_{ij}},\x_{EP} \rangle}{R_{c_{ij}}r_{c_{ij}}}+\alpha_{ij}^2 \right],
		 \end{align*} and  
	 \begin{align*} 
		\rho_{P_j}^2&=\left( \cufrac{\V_{ij}^P}{x_{P_j}}\right) ^2+\left( \cufrac{\V_{ij}^P}{y_{P_j}}\right) ^2+\left( \cufrac{\V_{ij}^P}{z_{P_j}}\right) ^2\\ &=\tfrac{\alpha_{ij}^4}{(1-\alpha_{ij}^2)^2}\left[1-\tfrac{2}{1-\alpha_{ij}^2}\tfrac{\langle \x_{c_{ij}},\x_{EP} \rangle}{R_{c_{ij}}r_{c_{ij}}}+\tfrac{1}{\alpha_{ij}^2} \right] 
		= \alpha_{ij}^2\rho_{E_i}^2.
		\end{align*}
		Hence, the function defined by \eqref{eq:value_Pout} satisfies the HJI-PDE \eqref{eq:ME2}, and is therefore the Value of the \texttt{1v1} RADG for $\x\in\mR_P$ with $\alpha_{ij}<1$. The optimal strategies of the players are obtained in feedback form by substituting the Value function in \eqref{eq:opt_con}.
	\end{proof}
	\begin{remark}
		The work by Garcia \cite{garcia_2020} considers a special case of the \texttt{1v1} RADG with $\alpha_{ij}=1$. In Theorem \ref{theorem:1v1_varspeed}, our work generalizes this result by considering all scenarios with $\alpha_{ij}<1$. This extension is non-trivial as the game geometry separating the dominance regions of the pursuer and the evader by is an Apollonius sphere instead of a plane as in the earlier work.
		\label{rem:alpha=1}
	\end{remark}
	\begin{theorem}[Evader winning region] Consider the \texttt{1v1} RADG with $\x\in\mR_E$ and $\alpha_{ij}\leq1$. The Value function is $\C^1$ and it is the solution of the HJI-PDE \eqref{eq:hji}. The Value function is given by 
		\begin{equation}
			\V_{ij}^E(\x)=-R_{P_{j}}+\frac{R_{E_{i}}}{\alpha_{ij}}, \label{eq:valueE}
		\end{equation}
		where $R_{P_{j}}=\vert\vert \x_{P_j} \vert\vert$, and $R_{E_{i}}=\vert\vert \x_{E_i} \vert\vert$.
		\label{theorem:1v1_varspeed_E} 
	\end{theorem}
	\begin{proof}
		Since $\x\in\mR_E$, the evader can reach the target faster than the pursuer. Thus, it might seem optimal to the evader to head directly towards the target in accordance to the result in Lemma \ref{theorem:prelim}. In this scenario, the time taken by the evader to reach origin is given by $t_{E_i}=\tfrac{\sqrt{x_{E_i}^2+y_{E_i}^2+z_{E_i}^2}}{U_i}=\tfrac{R_{E_i}}{U_i}$. During this time, $P_j$ covers a distance of $V_jt_{E_i}$ along the direction towards the terminal location of the evader, the origin. Hence, the terminal location of $P_j$ is given by 	$\x_{P_{j_f}}=\x_{P_j}-V_jt_{E_i}\tfrac{\x_{P_j}}{R_{P_j}}$.
		Then, the distance $||\x_{P_{j_f}}||$ gives the payoff, considering the objective in the evader winning region \eqref{eq:multi_Ecost}. This payoff now forms a guess for the Value function given by $\V_{ij}^E(\x)=-||\x_{P_{j_f}}||$ which equals the expression proposed in \eqref{eq:valueE}. The negative sign arises due to convention as the pursuer is considered to be the maximizing player in this article. 
		
		It remains to be checked if the proposed Value function \eqref{eq:valueE} satisfies the HJI-PDE which can be written as
		\begin{multline}
			\min\limits_{\u_i}\max\limits_{\v_j}\left(\cufrac{\V_{ij}^E}{x_{E_i}} u_{x_i}+\cufrac{\V_{ij}^E}{y_{E_i}}u_{y_i}+\cufrac{\V_{ij}^E}{z_{E_i}}u_{z_i}\right. \\
			\left. +\cufrac{\V_{ij}^E}{x_{P_j}}v_{x_j}+\cufrac{\V_{ij}^E}{y_{P_j}}v_{y_j}+\cufrac{\V_{ij}^E}{z_{P_j}}v_{z_j}\right)=0. \label{eq:ME1E}
		\end{multline}
		Thus, the optimal controls in terms of the states and the Value function can be written as 
		\begin{equation}
			\begin{aligned}
				\u_i^*&=-\tfrac{U_i}{\rho_{E_i}}\begin{bmatrix}
					\cufrac{\V_{ij}^E}{x_{E_i}}&\cufrac{\V_{ij}^E}{y_{E_i}}&\cufrac{\V_{ij}^E}{z_{E_i}}
				\end{bmatrix},\\  \v_j^*&=\tfrac{V_j}{\rho_{P_j}}\begin{bmatrix}
					\cufrac{\V_{ij}^E}{x_{P_j}}&\cufrac{\V_{ij}^E}{y_{P_j}}& \cufrac{\V_{ij}^E}{z_{P_j}}
				\end{bmatrix},\label{eq:opt_conE}
			\end{aligned}
		\end{equation}
		where $\rho_{E_i}=\sqrt{\left( \cufrac{\V_{ij}^E}{x_{E_i}}\right) ^2+\left( \cufrac{\V_{ij}^E}{y_{E_i}}\right) ^2+\left( \cufrac{\V_{ij}^E}{z_{E_i}}\right) ^2}$ and $\rho_{P_j}=\sqrt{\left( \cufrac{\V_{ij}^E}{x_{P_j}}\right) ^2+\left( \cufrac{\V_{ij}^E}{y_{P_j}}\right) ^2+\left( \cufrac{\V_{ij}^E}{z_{P_j}}\right) ^2}$. Substituting the optimal controls \eqref{eq:opt_conE} in the  HJI-PDE \eqref{eq:ME1E}, we get
		
		\begin{equation}
			-\alpha_{ij}\rho_{E_i}+\rho_{P_j}=0. \label{eq:ME2E}
		\end{equation}
		To verify if the proposed Value function \eqref{eq:valueE} satisfies the HJI-PDE \eqref{eq:ME2E}, we first determine the partial derivatives as		
		\begin{equation}
			\begin{aligned}
				\begin{bmatrix}\cufrac{\V_{ij}^E}{x_{E_i}}&\cufrac{\V_{ij}^E}{y_{E_i}}&\cufrac{\V_{ij}^E}{z_{E_i}}\end{bmatrix}&=\tfrac{1}{\alpha_{ij}}\tfrac{\x_{E_i}}{R_{E_i}},\\
				\begin{bmatrix}\cufrac{\V_{ij}^E}{x_{P_j}}&\cufrac{\V_{ij}^E}{y_{P_j}}&\cufrac{\V_{ij}^E}{z_{P_j}}\end{bmatrix}&=-\tfrac{\x_{P_j}}{R_{P_j}}. \label{eq:gradVE}
			\end{aligned}
		\end{equation}
		Using \eqref{eq:gradVE}, the LHS of the HJI-PDE \eqref{eq:ME2E} is determined as 		$-\alpha_{ij}\sqrt{\tfrac{1}{\alpha_{ij}^2}\left[\left(\tfrac{x_{E_i}}{R_{E_i}} \right)^2+\left( \tfrac{y_{E_i}}{R_{E_i}}\right)^2+\left( \tfrac{z_{E_i}}{R_{E_i}}\right)^2 \right] } +\sqrt{\left(\tfrac{x_{P_j}}{R_{P_j}} \right)^2+\left( \tfrac{y_{P_j}}{R_{P_j}}\right)^2+\left( \tfrac{z_{P_j}}{R_{P_j}}\right)^2}=-\alpha_{ij}\tfrac{1}{\alpha_{ij}}+1=0.$
		The function defined by \eqref{eq:valueE} satisfies the HJI-PDE and is, therefore, the Value function of the differential game. The optimal strategies of the players in feedback form are obtained by substituting the Value function in \eqref{eq:opt_conE}.
	\end{proof}
	
	\begin{remark}
		We note that Theorem \ref{theorem:1v1_varspeed_E} provides
		the Value function in the evader winning region. In the existing literature GoD analysis 
		is mostly performed exclusively in the pursuer winning region; see for example \cite{pachter_2018,yan_guarding_2022,garcia_2020,garcia_2021}. Here, we provide a comprehensive solution to the MRADG in 3D space by obtaining solutions to GoDs in both the pursuer and evader winning regions.
	\end{remark}
	
	\begin{figure}[h]  \centering 
		\subfloat[]{
			\includegraphics[width=0.15\textwidth,valign=t]{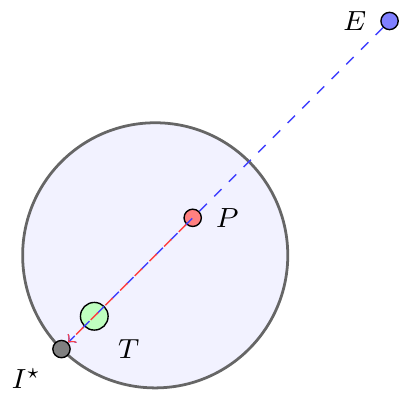}
			\label{fig:Rp_pathology}}\qquad
		\subfloat[]{
			\includegraphics[width=0.15\textwidth,valign=t]{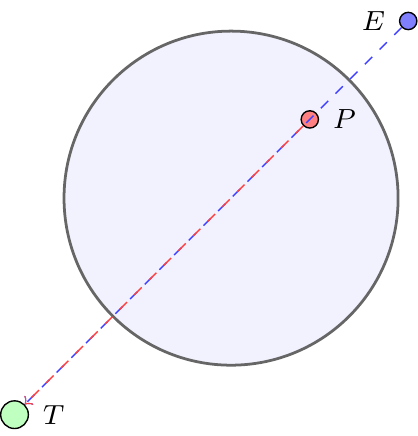}
			\label{fig:Re_pathology} }
		\caption{Pathological cases}
		\label{fig:pathological} 
	\end{figure}
	\begin{remark}
		The analytical characterization of the Value function for the \texttt{1v1} RADG has been provided in both the pursuer and evader winning regions only for $\alpha_{ij}\leq1$ (Theorem \ref{theorem:1v1_varspeed}, Remark \ref{rem:alpha=1} and Theorem \ref{theorem:1v1_varspeed_E}). However, the methodology proposed here cannot be extended to a situation where $\alpha_{ij}>1$.
 		To see this, consider a situation as shown in Figure \ref{fig:Rp_pathology}. Since the target lies within the Apollonius sphere along with the pursuer, this situation clearly corresponds to the pursuer winning. Following the established approach outlined in Theorem \ref{theorem:1v1_varspeed}, if both the agents head towards the optimal interception point $I^\star$, then they are bound to cross the target before reaching $I^\star$, which is clearly non-optimal. Now, consider a similar situation in the evader winning region as shown in Figure \ref{fig:Re_pathology}. If the solution approach presented in Theorem \ref{theorem:1v1_varspeed_E} is adopted, both the agents head directly towards the target, resulting in an interception on the boundary of the Apollonius sphere well before the evader can reach the target. Again, this behaviour is clearly non-optimal for the evader as it has the advantage of superior speed. Also, note that the illustrations are projections on the 2D space for clarity. 
	\end{remark}

	\section{Multiplayer Differential Game}
	\label{sec:multiplayer}
	In this section, we analyze a class of MRADGs using the results obtained from the previous section. Specifically, we solve the co-design problem by providing an optimal assignment scheme and an analytical characterization of the player winning regions and the Value function. 
	The class of MRADGs considered in this paper is characterized by the following assumption on players' interactions.
	\begin{assumption}
				\begin{enumerate}[label=(\roman*)]
			\item \label{itm:assumitem1}  A pursuer can capture at most one evader, and an evader is pursued by one pursuer.
			\item \label{itm:assumitem2} The pursuing team is atleast as large as the evading team.
			\item \label{itm:assumitem3} The pursuers commit to their assignments throughout the duration of the game.
		\end{enumerate}
	\label{assum:assumption1}
	\end{assumption}
	Item \ref{itm:assumitem1} is crucial, and using this, we show that there exists a linear-programming based optimal assignment scheme for matching the pursuers with the evaders. Item \ref{itm:assumitem2} is a natural choice following Item \ref{itm:assumitem1}, because if $n<m$ then, by Item \ref{itm:assumitem1}, at least $m-n$ evaders cannot be assigned to a pursuer, thus making the game outcome trivial.
		Item \ref{itm:assumitem3} implies that the pursuers strictly commit to their assignments throughout the game, \textit{i.e.}, $\mu_{ij}(t)=\mu_{ij}(0)\ \forall t\in[0,t_f]$. 
	Note that we do not make the assumption that all pursuers are faster than all evaders, as is often assumed in the existing literature; see \cite{yan_guarding_2022,pachter_2019,moll_2022,garcia_2020,garcia_2021}.
	\par	
	Next, we define a set of notations to aid in further analysis. First, the set of all \emph{feasible} assignments is denoted by
	\begin{align}
		\Sigma:=\big\{\bm{\mu}\in\{0,1\}^{m\times n}~\big|~
		\bm\mu\bm 1_n=\bm1_m,~ \bm\mu^T\bm 1_m\leq \bm1_n \big\}, \label{eq:feasibleset}
	\end{align}
	where $\bm1_k$ denotes a column one vector of size $k$.  
	
	We denote $\mu_{ij}$ as the element in $i^{th}$ row and $j^{th}$ column of any $\bm\mu\in\R^{3(m+n)}$. For every $\bm\mu\in\Sigma$, we define the associated index set as
	\begin{equation}
		A_{\bm\mu}:=\big\{ ij\in M\times N~\big|~\mu_{ij}=1 \big\}.
	\end{equation}
	It is worth mentioning that for every feasible assignment $\bm\mu\in\Sigma$, there is a unique associated index set. Further, any two distinct feasible assignments result in distinct index sets. 
	
	Next, we denote the set of all \emph{probabilistic feasible} assignments by
	\begin{align}
		\Gamma:=\big\{\bm{\gamma}\in[0,1]^{m\times n}~\big|~\bm\gamma\bm 1_n=\bm1_m,~ \bm\gamma^T\bm 1_m\leq \bm1_n \big\}. \label{eq:feasiblepset}
	\end{align}
	
	Denote $\bm\gamma_{ij}$ as the element in $i^{th}$ row and $j^{th}$ column of $\bm\gamma$. The payoff received by the pursuer $P_j$ when matched with an evader $E_i$ is now defined as
	\begin{align}
		a_{ij}(\x_{E_i},\x_{P_j})= \begin{cases} \mathcal V_{ij}^P(\x_{E_i},\x_{P_j});&  B_{ij}(\x_{E_i},\x_{P_j})> 0,~ \alpha_{ij}\leq1  \\ -L;&  \text{otherwise}. \end{cases} \label{eq:aval}
	\end{align}
	The pursuer $P_j$ assigned to evader $E_i$ receives a payoff equal to the Value of the game obtained from \eqref{eq:value_Pout}, if $\alpha_{ij}<1$, and $(\x_{E_i},\x_{P_j})$ lies in the pursuer winning region of the induced \texttt{1v1} RADG. For the $\alpha_{ij}=1$ case, the payoff follows from \cite{garcia_2020}. Notice that, if $\B_{ij}(\x)>0$, then the corresponding $\V_{ij}^P(\x)$ is strictly positive. In any other case, the pursuer incurs a cost equal to $L>0$. 
	
	Based on this payoff, the index set for some feasible assignment $\bm\mu\in\Sigma$ can be partitioned as follows:
	\begin{equation}
		A_{\bm\mu}^W:=\{ij\in A_{\bm\mu}~|~a_{ij}(\x_{E_i},\x_{P_j})>0 \} \text{ and } A_{\bm\mu}^L:=A_{\bm\mu}\setminus A_{\bm\mu}^W. \label{eq:ass_decompose}
	\end{equation}
	This partition distinguishes between the pairings where the pursuer wins and the ones where the pursuer loses.
	\par
	Now, using the payoff defined in \eqref{eq:aval}, the contribution of capturing an evader $E_i$ to the pursuer team is expressed as 
	\begin{align}
		\Psi_i((\x_{E_i},\x_P),\bm{\mu}):=\sum_{j\in N} a_{ij}(\x_{E_i},\x_{P_j})~\mu_{ij}.
		\label{eq:payoffcontribution}
	\end{align}
	Note that, if $\bm\mu\in\Sigma$ as defined in \eqref{eq:feasibleset}, then only one pursuer $P_j$ can be matched to an evader $E_i$. Consequently, the contribution of capturing $E_i$  is $\Psi_i((\x_{E_i},\x_{P_j}),\bm\mu)=a_{ij}(\x_{E_i},\x_{P_j})$.
	Now, the team payoff of the pursuers under the assignment $\bm\mu$ is given by 
	\begin{equation}
		\Psi(\x,\bm\mu):=\sum\limits_{i\in M}\Psi_i\left((\x_{E_i},\x_{P}),\bm\mu\right). \label{eq:assignment_payoff}
	\end{equation}
	
	We denote the set of pursuers who are faster than some evader $E_i$ and can capture this evader as $\hat N_i:=\{j\in N:\B_{ij}(\x_{E_i},\x_{P_j})>0,\ \alpha_{ij}\leq1\}$. Among the pursuers in $\hat N_i$, the maximum payoff generated by a pursuer who can capture $E_i$ is obtained as $\max_{j\in \hat{N}_i} \mathcal V_{ij}^P(\x_{E_i},\x_{P_j})$. Using this, the best possible payoff the pursuer team could generate during their engagement with the evader team is then obtained as
	\begin{align}
		L^\star(\x):=\sum_{i\in M}~ \max_{j\in \hat N_i} \mathcal V_{ij}^P(\x_{E_i},\x_{P_j}). \label{eq:lowerboundcost}
	\end{align}
	\begin{remark} 
		We note that $L^\star(\x)$   represents the maximum payoff achievable by the pursuer team when each evader is assigned to the pursuer that can yield the highest possible payoff. However, such an assignment need not be feasible.
	\end{remark} 
	
	%
	\subsection{Game of Kind}
	\label{subsec:gok}
	In this subsection, we characterize the winning regions for the pursuer and evader teams in the MRADG, based on the analysis in Section \ref{sec:1v1}. To this end, in the next theorem, we provide a solution to the optimal matching problem of pursuers to evaders in the MRADG by modeling it as a form of Shapley-Shubik assignment game \cite{shapley_1971}.
	
	\begin{theorem}
		The linear programming problem defined by
		\begin{equation}
			\Gamma^\star(\x):=\argmax\limits_{\bm\gamma\in\Gamma} \Psi(\x,\bm\gamma),\quad \x\in\R^{3(m+n)}
			\label{eq:lpsolution}
		\end{equation}
		satisfies $\Gamma^\star(\x)\subseteq\Sigma$, i.e., every assignment $\bm\gamma^\star\in\Gamma^\star(\x)$ is a feasible assignment. Starting from some $\x\in\R^{3(m+n)}$, if $L>L^\star(\x)$ then every optimal assignment $\bm\gamma^\star\in\Gamma^\star(\x)$ ensures that the least number of evaders reach the target. 
		\label{thm:lp}
	\end{theorem}
	\begin{proof}
		We note that the  linear program described in \eqref{eq:lpsolution} is a form of Shapley-Shubik assignment game \cite{shapley_1971}, which  has been shown to have integer solutions \cite{dantzig_1963}. Therefore, the optimal solutions to \eqref{eq:lpsolution} are indeed feasible assignments, i.e., $\Gamma^\star(\x)\subseteq\Sigma$, even though the search space encompasses all probabilistic feasible assignments. 
		
		Now, fix some $\G\in\Gamma^\star(\x)$ and assume there exists an assignment $\bm\mu\in\Sigma$ that allows $m_\mu$ evaders to reach the target. If $L>L^\star(\x)$, we claim that the optimal assignment $\bm\gamma^\star$ allows at most $m_\mu$ evaders to reach the target. Assume to the contrary that the optimal assignment allows $m_\G$ evaders to reach the target with $m_\G>m_\mu$. 
		The assignments $\bm\mu$ and $\bm\gamma^\star$ yield payoffs $\Psi(\x,\bm\mu)=\sum\limits_{ij\in A_\mu}a_{ij}(\x_{E_i},\x_{P_j})$ and $\Psi(\x,\bm\gamma^\star)=\sum\limits_{ij\in A_{\bm\gamma^\star}}a_{ij}(\x_{E_i},\x_{P_j})$ for the pursuer team respectively. Now, using the sets in \eqref{eq:ass_decompose}, we can write
		\begin{align*}
			&\Psi(\x,\bm\mu)-\Psi(\x,\bm\gamma^\star)=\sum\limits_{ij\in A_\mu}a_{ij}(\x_{E_i},\x_{P_j})-\sum\limits_{ij\in A_{\bm\gamma^\star}}a_{ij}(\x_{E_i},\x_{P_j})\\
			&=\Big( \sum\limits_{ij\in A_\mu^W}a_{ij}(\x_{E_i},\x_{P_j})+\sum\limits_{ij\in A_\mu^L}a_{ij}(\x_{E_i},\x_{P_j}) \Big) \\
			&\quad -\Big( \sum\limits_{ij\in A_{\bm\gamma^\star}^W}a_{ij}(\x_{E_i},\x_{P_j})+\sum\limits_{ij\in A_{\bm\gamma^\star}^L}a_{ij}(\x_{E_i},\x_{P_j}) \Big) \\
			&=\Big( \sum\limits_{ij\in A_\mu^W}a_{ij}(\x_{E_i},\x_{P_j})-\sum\limits_{ij\in A_{\bm\gamma^\star}^W}a_{ij}(\x_{E_i},\x_{P_j}) \Big)+(m_\G-m_\mu)L.
		\end{align*}
		By the definition of $L^\star(\x)$, we have that $0<\sum\limits_{ij\in A_\mu^W}a_{ij}(\x_{E_i},\x_{P_j})\leq L^\star(\x)\ \forall\bm\mu\in\Sigma$, and since $\bm\gamma^\star\in\Sigma$, we have $\sum\limits_{ij\in A_\mu^W}a_{ij}(\x_{E_i},\x_{P_j})-\sum\limits_{ij\in A_{\bm\gamma^\star}^W}a_{ij}(\x_{E_i},\x_{P_j})\geq -L^\star(\x)$. By assumption $L>L^\star(\x)$ and $m_\G>m_\mu$. This results in $m_\G-m_\mu\geq1$ as $m_\G$ and $m_\mu$ are both integers. Hence, we have that $\Psi(\x,\bm\mu)-\Psi(\x,\bm\gamma^\star)>0$ which is a contradiction to the fact that $\bm\gamma^\star$ optimizes the LP given by \eqref{eq:lpsolution}. Further, since the chosen $\bm\mu\in\Sigma$ is arbitrary, $m_\G\leq m_{\bm\mu},\ \forall\bm\mu\in\Sigma$. And, since the choice of $\G\in\Gamma^\star(\x)$ was arbitrary, any assignment that satisfies the optimization problem \eqref{eq:lpsolution} is  feasible and ensures that the least number of evaders reach the target. 
	\end{proof}

	\begin{corollary}
		Starting from some $\x\in\R^{3(m+n)}$ let $m_\G$ denote the number of evaders reaching the target under the assignment scheme $\G\in\Gamma^\star(\x)$. The number of evaders that reach the target under any assignment scheme maximizing \eqref{eq:lpsolution} is same, i.e.,	$m_{\bm\gamma_1^\star}=m_{\bm\gamma_2^\star},\ \forall \bm\gamma_1^\star,\ \bm\gamma_2^\star\in \Gamma^\star(\x)$.
		\label{cor:evader_num}
	\end{corollary}
	\begin{proof}
		 Consider $\bm\gamma_1^\star, \bm\gamma_2^\star \in \Gamma^\star(\x)$. As already shown in Theorem \ref{thm:lp}, we have $m_{\bm\gamma_i^\star}\leq m_{\bm\mu}\ \forall \bm\mu\in\Sigma$ for $i=1,2$. But since $\Gamma^\star(\x)\subseteq\Sigma$, so both the assignments $\bm\gamma_1^\star, \bm\gamma_2^\star \in\Sigma$. Thus, we have $m_{\bm\gamma_1^\star}\leq m_{\bm\gamma_2^\star}$ and similarly  $m_{\bm\gamma_2^\star}\leq m_{\bm\gamma_1^\star}$. Thus, we conclude that  $m_{\bm\gamma_1^\star}= m_{\bm\gamma_2^\star}$, i.e., all the assignments maximizing \eqref{eq:lpsolution} allow the same number of evaders to reach the target. 
	\end{proof}
	
	
	\begin{remark}
		We note that if the cost of failing to capture an evader  ($-L$)  is low enough, then the optimal assignment may allow more evaders to reach the target, even if there exists another feasible assignment allowing fewer evaders to win. Additionally, at every state $\x\in\R^{3(m+n)}$, the value $L^\star(\x)$ serves as a conservative lower bound for $L$. 
	\end{remark}
	
	\begin{remark}
		\label{rem:weighted}
		A majority of multiplayer differential game literature considers a maximum matching algorithm to extend the solution of a \texttt{1v1} interaction to the multiplayer game. However, a maximum matching algorithm is indifferent between any pair of assignments that result in the capture of the same number of evaders. It is possible however that one of these assignments results in a better payoff as a result of the underlying \texttt{1v1} interactions. To differentiate between such assignments, one must consider a weighted maximum matching algorithm with the weights representing some notion of the payoff received as a result of the underlying \texttt{1v1} interactions. To this end, the assignment problem considered in this work, is fundamentally a combinatorial challenge tasked with identifying the maximum weighted matching for a bipartite graph \cite{Kao_2016}. 
	\end{remark}
	
	\begin{remark}
		\label{rem:efficiency}
		As mentioned in Remark \ref{rem:weighted}, the assignment problem is fundamentally a maximum weighted matching algorithm. However, the assignment problem can also be formulated as a linear program; see \cite{shapley_1971}. While employing the linear program approach may not be the most expedient implementation, in practice it often performs significantly better than a brute-force method where all feasible assignments need to be enumerated, resulting in a factorial complexity. Further, it offers a systematic framework for analyzing multiplayer game situations. In particular, it aids in further analysis to determine the analytical solution to the GoK and the GoD. And, since the solutions to the assignment problem remain consistent regardless of the chosen solution approach, even faster algorithms \cite{Bijsterbosch_Volgenant_2010} can be utilized for obtaining these solutions.
	\end{remark}

	Next, using the assignment scheme obtained via the optimization problem \eqref{eq:lpsolution}, we device the analytical characterization of the pursuer and evader winning regions.

	\begin{definition}
		Consider the family of barrier functions $\B_\G:\R^{3(m+n)}\rightarrow\R$ defined by 
		\begin{equation}
			\B_\G(\x)=\min\limits_{i\in M}\Psi_i((\x_{E_i},\x_{P}),\bm{\gamma}^\star), \label{eq:nm_barrier}
		\end{equation}
		for every $\G\in\Gamma^\star(\x)$. Note, that $\B_\G$ is well-defined over its domain.
	\end{definition}
	
	\begin{remark}
		Note that since $\G\in\Gamma^\star(\x)$, the assignment $\G$ is also a function of $\x$. However, for the sake of brevity, this dependence is not explicitly shown (i.e., ideally it should have been $\G(\x)$). So, the term $\B_\G(\x)$ implies that the assignment $\G$ is an optimal assignment satisfying \eqref{eq:lpsolution} for the state $\x$. 
		
	\end{remark}
	
	\begin{lemma} Let Assumption \ref{assum:assumption1} hold. Consider an $\x\in\R^{3(m+n)}$ and some $\G\in\Gamma^\star(\x)$. Starting from $\x$, under the assignment scheme $\G$, the pursuer team wins if $\B_\G(\x)>0$ and the evader team wins if $\B_\G(\x)\leq0$. 
		\label{lemma:win}
	\end{lemma}
	\begin{proof}
		First consider the case that $\B_\G(\x)>0$. We claim that $a_{ij}(\x_{E_i},\x_{P_j})>0,\ \forall\ ij\in A_{\bm\gamma^\star}$. Assume to the contrary that there exists $\bar{i}\bar{j}\in A_{\bm\gamma^\star}$ such that $a_{\bar i\bar j}(\x_{E_i},\x_{P_j})\leq0$. As defined in \eqref{eq:aval}, it is possible only if $a_{\bar i\bar j}(\x_{E_{\bar i}},\x_{P_{\bar j}})=-L$ resulting in $\Psi_{\bar i}((\x_{E_{\bar i}},\x_P),\bm\gamma^\star)=\sum_{j\in N}a_{\bar i j}(\x_{E_{\bar i}},\x_{P_j})\bm\gamma_{\bar i j}^\star<0$. Thus, $\B_\G(\x)<0$, resulting in a contradiction. Hence, every pursuer assigned to an evader can intercept it outside the target. Further, the constraint $\bm\gamma^\star.\bm 1_n=1_m$ ensures that every evader has been assigned to some pursuer. From the previous two arguments, it is evident that the pursuer team wins.
		
		Now, consider the case when $\B_\G(\x)\leq0$. There exists $\bar i\in M$ such that $\Psi_{\bar i}((\x_{E_{\bar i}},\x_P),\bm\gamma^\star)=\sum_{j\in N}a_{\bar i j}(\x_{E_{\bar i}},\x_{P_j})\bm\gamma^\star_{\bar ij}\leq0$. This is again possible only if $a_{\bar i\bar j}(\x_{E_{\bar i}},\x_{P_{\bar j}})\leq0$ for some $\bar j\in N$ with $\bar i\bar j\in A_{\bm\gamma^\star}$. This implies that, under the assignment $\bm\gamma^\star$, at least one evader reaches the target before interception. Therefore, using Theorem \ref{thm:lp}, every $\bm\mu\in\Sigma$ allows at least one evader to reach the target. Hence, the evader team wins. 
	\end{proof}
	As proved in Lemma \ref{lemma:win}, starting from some $\x\in\R^{3(m+m)}$, and choosing an assignment from the set of optimal assignments $\Gamma^\star(\x)$ lets us determine the winner of the interaction based on the corresponding barrier function. However, in the case that $|\Gamma^\star(\x)|>1$, there are multiple possible assignments of pursuers to evaders yielding the same payoff \eqref{eq:assignment_payoff}. In that case, does a choice of assignment from $\Gamma^\star(\x)$ alter the consequence of the game? We answer this question in Lemma \ref{lemma:invariance} and show that the result of the multiplayer game starting from some $\x\in\R^{3(m+n)}$ is independent of the choice of assignment from $\Gamma^\star(\x)$.

	\begin{lemma}[Invariance Property]
		The sets $\{\x\in\R^{3(m+n)}:\B_\G(\x)>0\ \text{ for any } \G\in\Gamma^\star(\x)\}$ and $\{\x\in\R^{3(m+n)}:\B_\G(\x)\leq0\ \text{ for any }\G\in\Gamma^\star(\x)\}$ are well-defined and form a partition of the global state space $\x\in\R^{3(m+n)}$.
		\label{lemma:invariance} 
	\end{lemma}
	\begin{proof}
		To prove that the set $\{\x\in\R^{3(m+n)}:\B_\G(\x)>0\ \text{ for any } \G\in\Gamma^\star(\x)\}$ is well-defined, we need to show that for any $\x\in\R^{3(m+n)}$, then $\B_{\bm\gamma_1^\star}(\x)>0$ if and only if $\B_{\bm\gamma_2^\star}(\x)>0$ for any distinct $\bm\gamma_1^\star,\bm\gamma_2^\star\in\Gamma^\star(\x)$. Also note that if $|\Gamma^\star(\x)|=1$ then the set $\{\x\in\R^{3(m+n)}:\B_\G(\x)>0\ \text{ for any } \G\in\Gamma^\star(\x)\}$ is trivially well-defined.
 		
		Now, consider some $\x\in\R^{3(m+n)}$. Let $\bm\gamma_1^\star,\bm\gamma_2^\star\in\Gamma^\star(\x)$ be two distinct assignments and assume that $\B_{\bm\gamma_1^\star}(\x)>0$. Then, from Lemma \ref{lemma:win} we have that the pursuer team wins starting from $\x$ under the assignment $\bm\gamma_1^\star$. This implies that all the evaders are intercepted away from the target, i.e., the number of evaders reaching the target $m_{\bm\gamma_1^\star}=0$. But, from Corollary \ref{cor:evader_num}, we have $m_{\bm\gamma_1^\star}=m_{\bm\gamma_2^\star}\ \forall \bm\gamma_1^\star, \bm\gamma_2^\star\in\Gamma^\star(\x)$. So, even under the assignment scheme $\bm\gamma_2^\star$ all the evaders are captured away from the target. Then, starting from $\x$, the pursuer team wins even under the assignment $\bm\gamma_2^\star$. Suppose $\B_{\bm\gamma_2^\star}(\x)\leq0$, then by Lemma \ref{lemma:win}, the evader team wins, which is a contradiction. Hence, we have $\B_{\bm\gamma_2^\star}(\x)>0$. Using a similar reasoning, it can be shown that if $\B_{\bm\gamma_2^\star}(\x)>0$ then $\B_{\bm\gamma_1^\star}(\x)>0$. Thus, we prove that the set $\{\x\in\R^{3(m+n)}:\B_\G(\x)>0\ \text{ for any } \G\in\Gamma^\star(\x)\}$ is well-defined. 
		
		Since, a barrier function is well-defined, for some $\x\in\R^{3(m+n)}$ if $\B_{\bm\gamma_1^\star}$ is not positive, then it must be $\leq0$. But this implies that for a state $\x\in\R^{3(m+n)}$, $\B_{\bm\gamma_1^\star}(\x)\leq0$ if and only if $\B_{\bm\gamma_2^\star}(\x)\leq0$ for any distinct $\bm\gamma_1^\star,\bm\gamma_2^\star\in\Gamma^\star(\x)$. Thus, the set $\{\x\in\R^{3(m+n)}:\B_\G(\x)\leq0\ \text{ for any }\G\in\Gamma^\star(\x)\}$ is well-defined. And the sets $\{\x\in\R^{3(m+n)}:\B_\G(\x)>0\ \text{ for any }\G\in\Gamma^\star(\x)\}$ and $\{\x\in\R^{3(m+n)}:\B_\G(\x)\leq0\ \text{ for any }\G\in\Gamma^\star(\x)\}$ clearly form a partition of the global state space. 
		
	\end{proof}

	Now, using the Lemmas \ref{lemma:win} and \ref{lemma:invariance}, we provide an analytical characterization of the pursuer and evader winning regions as the super and sub-level sets of the family of barrier functions introduced in \eqref{eq:nm_barrier}.
	\begin{theorem} Let Assumption \ref{assum:assumption1} hold. The pursuer and evader winning regions of the MRADG are obtained as the sets $\mR_P:=\{\x\in\R^{3(m+n)}:\B_\G(\x)>0\ \text{ for any } \G\in\Gamma^\star(\x)\}$ and $\mR_E:=\{\x\in\R^{3(m+n)}:\B_\G(\x)\leq0\ \text{ for any } \G\in\Gamma^\star(\x)\}$ respectively.
		\label{thm:multibarrier}
	\end{theorem}
	\begin{proof}
		First, we recall Lemma \ref{lemma:invariance} to note that the sets $\{\x\in\R^{3(m+n)}:\B_\G(\x)>0\ \text{ for any } \G\in\Gamma^\star(\x)\}$ and $\{\x\in\R^{3(m+n)}:\B_\G(\x)\leq0\ \text{ for any } \G\in\Gamma^\star(\x)\}$ are well-defined and form a partition of the global state space $\x\in\R^{3(m+n)}$.
		
		Consider $\x\in\{\x\in\R^{3(m+n)}:\B_\G(\x)>0\ \text{ for any } \G\in\Gamma^\star(\x)\}$. Since $\B_\G(\x)>0$ for any $\G\in\Gamma^\star(\x)$, from Lemma \ref{lemma:win} we know that the pursuer team wins. Thus, $\{\x\in\R^{3(m+n)}:\B_\G(\x)>0\ \text{ for any } \G\in\Gamma^\star(\x)\}\subseteq\mR_P$. Similarly for some  $\x\in\{\x\in\R^{3(m+n)}:\B_\G(\x)\leq0\ \text{ for any } \G\in\Gamma^\star(\x)\}$, using Lemma \ref{lemma:win} we know that the evader team wins. Thus, $\{\x\in\R^{3(m+n)}:\B_\G(\x)\leq0\ \text{ for any } \G\in\Gamma^\star(\x)\}\subseteq\mR_E$.
		
		Now, consider the opposite direction and let $\x\in\mR_P$. Suppose for contradiction $\x\in\{\x\in\R^{3(m+n)}:\B_\G(\x)\leq0\ \text{ for any } \G\in\Gamma^\star(\x)\}$. Then by Lemma \ref{lemma:win}, the evader team wins, i.e., $\x\in\mR_E$ which is a contradiction. Thus, $\mR_P\subseteq\{\x\in\R^{3(m+n)}:\B_\G(\x)>0\ \text{ for any } \G\in\Gamma^\star(\x)\}$. Using a similar reasoning, we show that $\mR_E\subseteq\{\x\in\R^{3(m+n)}:\B_\G(\x)\leq0\ \text{ for any } \G\in\Gamma^\star(\x)\}$. From the earlier part of the proof, it is then evident that $\mR_P=\{\x\in\R^{3(m+n)}:\B_\G(\x)>0\ \forall\G\in\Gamma^\star(\x)\}$ and $\mR_E=\{\x\in\R^{3(m+n)}:\B_\G(\x)\leq0\ \forall\G\in\Gamma^\star(\x)\}$.
	\end{proof}
	
	Now, we consider an example to illustrate the existence of multiple optimal assignments for both the pursuer and evader winning regions. 
	
	\begin{example}
		\label{ex:nonuniqe_gamma}
		Consider an MRADG involving 3 pursuers and 3 evaders, with the payoff matrices for two initial state positions ($\bar{\x}$ and $\x$)  given by
		\begin{equation}
			a(\bar\x)=\begin{bmatrix}
				3.23 & 1.34 & 2.21\\
				3.66 & 1.77 & 1.67\\
				2.89 & 3.24 & 9.56
			\end{bmatrix},\quad
			a(\x)=\begin{bmatrix}
				4.024 & -L & -L\\
				9.72 & -L & -L\\
				1.38 & -L & -L
			\end{bmatrix}.
		\end{equation}
		Consider first the payoff $a(\bar\x)$ and the resulting $\Gamma^\star(\bar\x)$ is $\{\bar{\bm\gamma}_{1}^\star,\bar{\bm\gamma}_{2}^\star\}$ which can be characterized by the index sets $A_{\bar{\bm\gamma}_{1}^\star}=\{11,22,33\}$ and $A_{\bar{\bm\gamma}_{2}^\star}=\{12,21,33\}$. As $\B_{\bar{\bm\gamma}_1^\star}(\bar\x)>0$ and $\B_{\bar{\bm\gamma}_2^\star}(\bar\x)>0$, this situation corresponds to the pursuer winning region, and the pursuer team receives equal payoff and is indifferent between the assignments $\bar{\bm\gamma}_{1}^\star$ and $\bar{\bm\gamma}_{2}^\star$. Now, consider the payoff $a(\x)$ and let the resulting $\Gamma^\star(\x)$ be $\{\bm\gamma_{1}^\star,\bm\gamma_{2}^\star\}$ which can be characterized by $A_{\bm\gamma_{1}^\star}=\{21,12,33\}$  and $A_{\bm\gamma_{2}^\star}=\{21,32,13\}$. Here, since $\B_{\bm\gamma_1^\star}(\x)\leq0$ and $\B_{\bm\gamma_2^\star}(\x)\leq0$, this situation is in the evader winning region, and similar to the earlier case, there are two assignments resulting in equal team payoff ($\Psi$). It is worth noting that for any situation in the evader winning region, if some assignment satisfying \eqref{eq:lpsolution} has $\bar m>1$ number of pairings resulting in the evader winning, then there are atleast $\bar m!$ other assignments which result in same payoff. 
		Now note that in the assignments corresponding to $a(\x)$, the equality in payoff for the two assignments arises as the pairings $\{12,33\}$ and $\{32,13\}$ result in a net payoff of $-2L$ for the team, along with the common of pairing of $\{21\}$. However, if a matching results in the corresponding evader winning, then the pursuer team should get the payoff corresponding to the $\texttt{1v1}$ RADG in the evader winning region. Thus, the actual payoff obtained by the pursuer team should be different from the one given by $a(\x)$ which is contrived to ensure that evaders are captured before they succeed. 
		
	\end{example}

	\subsection{Game of Degree}
	\label{subsec:god}
	Having solved the GoK for the MRADG, starting from any point $\x\in \R^{3(n+m)}$ in the global state space we now study the GoD for the MRADG in the pursuer winning region. 
	
	\begin{theorem}[Pursuer team winning region] Let Assumptions \ref{assum:assumption1} hold.
		Consider the  MRADG for $\x\in\mR_P$. If $L>L^\star(\x)$, then the Value function is $C^1$ (except at the dispersal surfaces i.e., there are multiple optima for the LP in \eqref{eq:lpsolution}) and it is the solution of the HJI-PDE \eqref{eq:hji}. The Value function is given by 
		\begin{equation}
			\V(\x)= \Psi(\x,\bm\gamma^\star), \label{eq:multi_value}
		\end{equation}
		where $\bm\gamma^\star\in \Gamma^\star(\x)$ satisfies \eqref{eq:lpsolution}. The optimal strategies of all the players are obtained from \eqref{eq:opt_con}.
		\label{thm:valueP}
	\end{theorem}
	\begin{proof}
		Let $\G\in\Gamma^\star(\x)$. Since $\x\in\mR_P$, we have $\B_\G(\x)>0$, and consequently $a_{ij}(\x_{E_i},\x_{P_j})>0,\ \forall ij\in A_{\bm\gamma^\star}$ as shown in Theorem \ref{thm:multibarrier}. But, as defined in \eqref{eq:aval}, $a_{ij}(\x_{E_i},\x_{P_j})$ takes positive values only when $a_{ij}(\x_{E_i},\x_{P_j})=\V_{ij}^P(\x_{E_i},\x_{P_j})$. Thus, the proposed Value function \eqref{eq:multi_value} can be written as 
		\begin{equation}
			\begin{aligned}
				\V(\x)=\Psi(\x,\bm\gamma^\star)=\sum\limits_{ij\in A_{\bm\gamma^\star}}a_{ij}(\x_{E_i},\x_{P_j})=\sum\limits_{ij\in A_{\bm\gamma^\star}}\V_{ij}^P(\x_{E_i},\x_{P_j}). \label{eq:multi_value_simplify}
			\end{aligned}
		\end{equation}
		The proposed Value of the multiplayer game is thus obtained as the sum of Values associated with individual \texttt{1v1} RADGs dictated by the optimal assignment. 
		
		We must now verify that the proposed Value function \eqref{eq:multi_value} satisfies the HJI-PDE \eqref{eq:hji}:
		\begin{align*}
			&\min\limits_{\u}\max\limits_{\v}\sum\limits_{ij\in A_{\G}}\Big\langle\Big[\tfrac{\partial \V_{ij}^P}{\partial\x_{E_i}}\ \tfrac{\partial\V_{ij}^P}{\partial\x_{P_j}}\Big],[\u_i\ \v_j] \Big\rangle \\
			&=\sum\limits_{ij\in A_\G} \min\limits_{\u_i}\max\limits_{\v_j}\left[\cufrac{\V_{ij}^P}{x_{E_i}} u_{x_i}+\cufrac{\V_{ij}^P}{y_{E_i}}u_{y_i}\right. \\
			& \left.+\cufrac{\V_{ij}^P}{z_{E_i}}u_{z_i}+\cufrac{\V_{ij}^P}{x_{P_j}}v_{x_j}+\cufrac{\V_{ij}^P}{y_{P_j}}v_{y_j}+\cufrac{\V_{ij}^P}{z_{P_j}}v_{z_j}\right]=0.
		\end{align*}

		The separability of the expression in terms of the individual controls of all the players allows the interchange of the summation with the minmax operator. Hence, the Value function proposed in \eqref{eq:multi_value} satisfies the HJI-PDE and provides the solution to the MRADG. The closed loop optimal controls can also be obtained from the Value function as given in \eqref{eq:opt_con}. Finally the dispersal surface is defined as the subset of state space where $\Gamma^\star(\x)$ has more than one element. 
	\end{proof}
	\begin{remark}
		The dispersal surface is a set of states starting from which, there are multiple equally optimal strategies for each team. To move away from the dispersal surface, each team can choose an assignment scheme randomly from the available optimal options in $\Gamma^\star(\x)$ at the instant the game begins. After an infinitesimal amount of time, the state moves out of the dispersal surface and the optimal strategies become fixed. However, this initial choice may favor one team and disadvantage their opponents; see \cite[Chapter 6]{isaacs_1965}.
	\end{remark}

	As seen in Theorem \ref{thm:valueP}, the solution to the GoD in the pursuer winning region can be determined using the assignment obtained from the optimization problem \eqref{eq:lpsolution}. However, this optimal assignment need not be the best solution in the evader winning region, as shall be illustrated in Example \ref{ex:refine}.

	\begin{example}
		\label{ex:refine}
		Now, continuing with Example \ref{ex:nonuniqe_gamma}, consider the payoff matrix $a(\x)$ given by
		\begin{equation}	
			a(\x)=\begin{bmatrix}
				4.024 & -L & -L\\
				9.72 & -L & -L\\
				1.38 & -L & -L
			\end{bmatrix}.
		\end{equation}
		As seen earlier, the optimal assignments are given by $\Gamma^\star(\x)=\{\bm\gamma_{1}^\star,\bm\gamma_{2}^\star\}$ which can be characterized by $A_{\bm\gamma_{1}^\star}=\{21,12,33\}$  and $A_{\bm\gamma_{2}^\star}=\{21,32,13\}$. Consider the assignment $\bm\gamma_1^\star$. It is natural for the pursuer team to receive a payoff based on the \texttt{1v1} interactions induced by the assignment which for this case results in $\V_{21}^P(\x_{E_2},\x_{P_1})+\V_{12}^E(\x_{E_1},\x_{P_2})+\V_{33}^E(\x_{E_3},\x_{P_3})=3.21$. On the other hand, choosing the assignment $\bm\gamma_{2}^\star$ results in a net payoff of $2.58$ using a calculation similar to the previous assignment. Clearly, there is a motive for the pursuer team to prefer the assignment $\bm\gamma_{1}^\star$ over $\bm\gamma_{2}^\star$. Hence, there is a need of further refining the set $\Gamma^\star(\x)$ based on a new payoff scheme for the evader winning region. This particular example has been considered with simulations in Example \ref{ex:Re}.
	\end{example}
	\begin{definition}
		Consider a matrix valued function $V:\R^{3(m+n)}\rightarrow\R^{m\times n}$ where each element is defined as 
		\begin{equation}
			V_{ij}(\x_{E_i},\x_{P_j})=\begin{cases}
				\V_{ij}^P(\x_{E_i},\x_{P_j}); \quad &\B_{ij}(\x_{E_i},\x_{P_j})>0,~\alpha_{ij}\leq1\\
				\V_{ij}^E(\x_{E_i},\x_{P_j}); \quad &\B_{ij}(\x_{E_i},\x_{P_j})\leq0,~ \alpha_{ij}\leq1\\
				-L; \quad &\text{otherwise}.
			\end{cases}
		\end{equation}
		This matrix is termed as the Value matrix denoted by $V(\x)$.
		\label{def:Vmat}
	\end{definition}
	
	\begin{lemma}[Refinement of $\Gamma^\star$]
		Consider $\x\in\R^{3(m+n)}$ and let $\Gamma^\star(\x)$ satisfy \eqref{eq:lpsolution}. Using the Value matrix $(V(\x))$ from Definition \ref{def:Vmat}, define a subset $\Theta^\star(\x)\subseteq\Gamma^\star(\x)$ as an optimization problem given by 
		\begin{equation}
			\Theta^\star(\x):=\argmax\limits_{\bm\theta\in\Gamma^\star(\x)}\sum\limits_{ij\in A_\theta}V_{ij}(\x_{E_i},\x_{P_j}). \label{eq:refine}
		\end{equation}
		Consider $\bar L^\star(\x)$ defined by 
		\begin{equation}
			\bar L^\star(\x)=2\sum\limits_{i\in M}\max\limits_{j\in \bar N_i}|V_{ij}(\x_{E_i},\x_{P_j})|,
		\end{equation}
		where $\bar N_i=\{j\in N:\alpha_{ij}\leq1\}$. Starting from some $\x\in\R^{3(m+n)}$, if $L>\bar L^\star(\x)$, then every $\Th\in\Theta^\star(\x)$ ensures that $|\{ij\in A_\Th:\alpha_{ij}>1\}|$ is minimized.
		\label{lem:refine}
	\end{lemma}
	\begin{proof}
		Fix some $\Th\in\Theta^\star(\x)$ and let $m_\Th=|\{ij\in A_\Th:\alpha_{ij}>1\}|$. Now, assume the existence of $\bm\mu\in\Gamma^\star(\x)$ such that $m_\mu:=|\{ij\in A_\mu:\alpha_{ij}>1\}|<m_\Th$. Let $A_\mu^{WE}:=\{ij\in A_{\bm\mu}:\alpha_{ij}\leq1\}$ and $A_\mu^{LE}:=\{ij\in A_{\bm\mu}:\alpha_{ij}>1\}$ for every $\mu\in\Gamma^\star$. Using these sets, similar to the proof in Theorem \ref{thm:lp}, consider the difference:
		\begin{align*}
			&\sum\limits_{ij\in A_\mu}V_{ij}(\x_{E_i},\x_{P_j})-\sum\limits_{ij\in A_\Th}V_{ij}(\x_{E_i},\x_{P_j})\\
			&= \left(  \sum\limits_{ij\in A_\mu^{WE}}V_{ij}(\x_{E_i},\x_{P_j})-\sum\limits_{ij\in A_\Th^{WE}}V_{ij}(\x_{E_i},\x_{P_j})\right) +(m_\Th-m_\mu)L
		\end{align*}

		By the definition of $\bar L^\star(\x)$, we have that $|\sum\limits_{ij\in A_\mu^{WE}}V_{ij}(\x_{E_i},\x_{P_j})|\leq \frac{\bar L^\star(\x)}{2}\ \forall \bm\mu\in\Gamma^\star(\x)$. By assumption, we have $L>\bar L^\star(\x)$ and $m_\Th>m_\mu$. This results in $m_\Th-m_\mu\geq1$ as $m_\Th$ and $m_\mu$ are both integers. Thus, the above difference can now be written as 
		\begin{align*}
			\left( \sum\limits_{ij\in A_\mu^{WE}}V_{ij}(\x_{E_i},\x_{P_j})+\frac{L}{2}\right) +\left(- \sum\limits_{ij\in A_\Th^{WE}}V_{ij}(\x_{E_i},\x_{P_j})+\frac{L}{2}\right)>0.
		\end{align*}
		But, this is a contradiction to the fact that $\Theta^\star(\x)$ satisfies \eqref{eq:refine}. Hence $m_\Th\leq m_\mu\ \forall \bm\mu\in\Gamma^\star(\x)$, thus proving the Lemma. 
	\end{proof}
		
	\begin{remark}
		Since $\Theta^\star(\x)\subseteq\Gamma^\star(\x)$, as per Theorem \ref{thm:multibarrier} all the assignments in $\bm\theta^\star\in\Theta^\star(\x)$ construct the same partitions for pursuer and evader winning regions using the barrier function $B_{\bm\theta^\star}$ as defined in \eqref{eq:nm_barrier}. Further, as shown in Lemma \ref{lem:refine}, if $L>\bar L^\star(\x)$ then the refinement provides assignments that satisfies \eqref{eq:refine} while ensuring that the least number of pairings occur where the resulting interaction cannot be predicted. 
	\end{remark}
		
	
	\begin{theorem}[Evader team winning region] Let Assumptions \ref{assum:assumption1} hold. Consider the  MRADG for $\x\in\mR_E$ and $L>\bar L^\star(\x)$. The Value function is $C^1$ (except at the dispersal surfaces i.e., there are multiple elements in the set $\Theta^\star$) and it is the solution of the HJI-PDE \eqref{eq:hji} if the set $\{ij\in A_{\Th}: \alpha_{ij}>1\}$ is empty. The Value function is given by 
		\begin{equation}
			\V(\x)= \sum\limits_{ij\in A_{\Th}}V_{ij}(\x_{E_i},\x_{P_j}), \label{eq:multi_value_E}
		\end{equation}
		where $\Th\in \Theta^\star(\x)$. The optimal strategies of all the players are obtained from \eqref{eq:opt_con} and \eqref{eq:opt_conE}.
		\label{thm:valueE}
	\end{theorem}
	\begin{proof}
		Note that the set $\{ij\in A_\Th:\alpha_{ij}>1\}$ has the same cardinality independent of the choice of $\Th\in\Theta^\star(\x)$ which can be proved in a manner similar to Lemma \ref{lemma:invariance}, using the result shown in Lemma \ref{lem:refine}. 
		
		Since $\x\in\mR_E$, there is atleast one evader that cannot be intercepted by the assigned pursuer. Following the definitions of $A_{\Th}^W$ and $A_{\Th}^L$ from \eqref{eq:ass_decompose}, 
		and due to the emptiness of the set $\{ij\in A_{\Th}:\alpha_{ij}>1\}$, we can write the proposed Value function \eqref{eq:multi_value_E} as 
		\begin{equation}
			\begin{aligned}
				\V(\x)=\sum\limits_{ij\in A_{\Th}^W}\V_{ij}^P(\x_{E_i},\x_{P_j})+\sum\limits_{ij\in A_{\Th}^L}\V_{ij}^E(\x_{E_i},\x_{P_j}).
			\end{aligned}
		\end{equation}
		We must now verify that the proposed Value function \eqref{eq:multi_value_E} satisfies the HJI equation:
		\begin{align*} 
			&\min\limits_{\u}\max\limits_{\v}\left[ \sum\limits_{ij\in A_{\Th}^W}\Big\langle\Big[\tfrac{\partial \V_{ij}^P}{\partial\x_{E_i}}\ \tfrac{\partial\V_{ij}^P}{\partial\x_{P_j}}\Big],[\u_i\ \v_j] \Big\rangle \right. \\
			& \qquad\qquad\left.+ \sum\limits_{ij\in A_{\Th}^L}\Big\langle\Big[\tfrac{\partial \V_{ij}^E}{\partial\x_{E_i}}\ \tfrac{\partial\V_{ij}^E}{\partial\x_{P_j}}\Big],[\u_i\ \v_j] \Big\rangle \right] \\
			&=\sum\limits_{ij\in A_{\Th}^W} \min\limits_{\u_i}\max\limits_{\v_j}\left[\cufrac{\V_{ij}^P}{x_{E_i}} u_{x_i}+\cufrac{\V_{ij}^P}{y_{E_i}}u_{y_i}\right.\\
			&\qquad \left.+\cufrac{\V_{ij}^P}{z_{E_i}}u_{z_i}+\cufrac{\V_{ij}^P}{x_{P_j}}v_{x_j}+\cufrac{\V_{ij}^P}{y_{P_j}}v_{y_j}+\cufrac{\V_{ij}^P}{z_{P_j}}v_{z_j}\right] \end{align*}
		\begin{align*} 
			&+\sum\limits_{ij\in A_{\Th}^L} \min\limits_{\u_i}\max\limits_{\v_j}\left[\cufrac{\V_{ij}^E}{x_{E_i}} u_{x_i}+\cufrac{\V_{ij}^E}{y_{E_i}}u_{y_i}\right. \\
			&\left.+\cufrac{\V_{ij}^E}{z_{E_i}}u_{z_i}+\cufrac{\V_{ij}^E}{x_{P_j}}v_{x_j}+\cufrac{\V_{ij}^E}{y_{P_j}}v_{y_j}+\cufrac{\V_{ij}^E}{z_{P_j}}v_{z_j}\right]=0. 
		\end{align*}
		Again, due to the separability of the expression in terms of the individual controls of all the players allows the interchange of the summations with the minmax operator. Hence, the Value function proposed in \eqref{eq:multi_value_E} and provides the solution to the MRADG. The closed loop controls can also be obtained from the Value function as given in \eqref{eq:opt_con} and \eqref{eq:opt_conE} depending on the result of the interaction of the agents along with the players they were assigned to. \par
		Similar to the reasoning in Theorem \ref{thm:valueP}, a dispersal surface arises in this scenario if the set $\Theta^\star(\x)$ is nonempty. Again, in such a scenario, the game can move away from the initial surface by an initial random choice for an infinitesimal amount of time. 
	\end{proof}

	\section{Numerical Illustrations}
	\label{sec:num}
	We consider a few representative MRADG examples to illustrate our solution approach. The computations were done in MATLAB 2022b on a workstation PC with a Core i9-13900K processor and a memory of 128GB. 
	\begin{example}
		\label{ex:efficiency}
		We illustrate the efficiency of our linear programming based optimal assignment by using the MATLAB \texttt{linprog} function. We compare our results with a brute-force algorithm (as discussed in Remark \ref{rem:efficiency}); see Table \ref{tab:Table1}. This table shows the average time required to compute assignments via brute-force and linear programming (in seconds) for different values of $n$ and $m$. The last column shows the required computational time for coordinate and distance calculations (in milliseconds). Although the simplex algorithm has an exponential time complexity, it can significantly outperform a brute-force method in most real-world cases, as evident in Table \ref{tab:Table1}.  
		\begin{table}[h]
			\caption{Average computation time for various multiplayer cases. NA indicates that the set of feasible assignments cannot be generated by the brute-force method due to memory limitations.}
			\label{tab:Table1}
			\centering
			{\setlength{\tabcolsep}{5pt} 
				\renewcommand{\arraystretch}{1} { 
					\begin{tabular}{| m{4em} | m{4em} | m{4em} | m{5em} |}
						\hline 
						Player number ($n,m$) & Brute Force Assignment(s) & LP-Based Assignment(s) &Coordinates/ Distance(ms) \\ 
						\hline \vspace{1.5pt}
						3,3 & 0.0188 & 0.0019 & 0.0074\\
						\hline 	 \vspace{1.5pt}
						7,5 & 0.0953 & 0.0022 & 0.0118 \\
						\hline \vspace{1.5pt}
						10,8 & 0.3439 & 0.0018 & 0.0178\\
						\hline \vspace{1.5pt}
						11,7 & 0.3264 & 0.0021 & 0.0182\\
						\hline  \vspace{1.5pt}
						12,10 & 59.561 & 0.0020 & 0.0285\\
						\hline  \vspace{1.5pt}
						20,15 & NA & 0.0020 & 0.0306 \\
						\hline \vspace{1.5pt} 
						50,40 & NA & 0.0039 & 0.0787\\
						\hline \vspace{1.5pt} 
						100,100 & NA & 0.0231 & 0.2032 \\
						\hline
					\end{tabular}
			} }
		\end{table} 
	\end{example}
	\begin{example}
		\label{ex:Rp}
		Consider an MRADG with the origin as the target with 3 pursuers and 3 evaders. The pursuers are situated at  $\x_{P_1}=(-6.77,-2.95,0.01),\ \x_{P_2}=(-3.34,-3.96,-3.33)$ and $\x_{P_3}=(4.76,-13.35,-0.61)$, while the evaders are situated at $\x_{E_1}=(4.92,-7.91,4.43),\ \x_{E_2}=(-8.07,2.73,-5.91)$ and $\x_{E_3}=(-6.73,-10.65,-12.49)$. The speeds of the pursuers are $V_{P_1}=1.71,\ V_{P_2}=2.23$ and $V_{P_3}=2.28$, while those of the evaders are $U_{E_1}=1.69,\ U_{E_2}=1.01$ and $U_{E_3}=1.84$. The pursuing team wins in this particular case and the optimal assignment is given by the index set $\{12,21,33\}$. The resulting Value of the game is $18.63$. The optimal trajectories for all the agents can be seen in Figure \ref{fig:optimal}. Note that in this setting there is a subset of pursuers who are not superior to all the evaders, namely $P_1$ is slower than $E_3$. Furthermore, $E_1$ cannot be captured by $P_3$ , and $E_3$ cannot be captured by $P_1$. 
		This particular example also serves as an illustration for the sensitivity of the assignment to the value of $L$. First we note that for the given initial positions $L^\star=23.84$. If $L>L^\star$, then the optimal assignment is as given above. However, for values of $L\leq L^\star$ the optimal assignment may not have the property of maximum capture of evaders. To illustrate this issue, consider $L=1$, then the optimal assignment provided by the LP in \eqref{eq:lpsolution} is characterized by $\{13,21,32\}$. Despite obtaining a payoff larger than the Value of the game, the assignment fails as $E_1$ and $E_2$ are assigned to pursuers who cannot capture them. Thus, $L$ must be large enough to ensure that the pursuer team goal of capturing all evaders is kept as a priority. Here, we also illustrate the robustness of the optimal state feedback strategies. Supposing that the evaders decided not to follow the obtained optimal strategies and instead decided to head straight to the origin with the pursuers committed to their optimal strategies, then the ensuing nonoptimal play results in a higher Value of $20.26$ for the game, favoring the pursuers. These trajectories are shown in Figure \ref{fig:nonoptimal}.
	\end{example}
	\begin{figure} 
		\centering
		\begin{subfigure}[h]{0.24\textwidth}
			\centering
			\includegraphics[scale=.5]{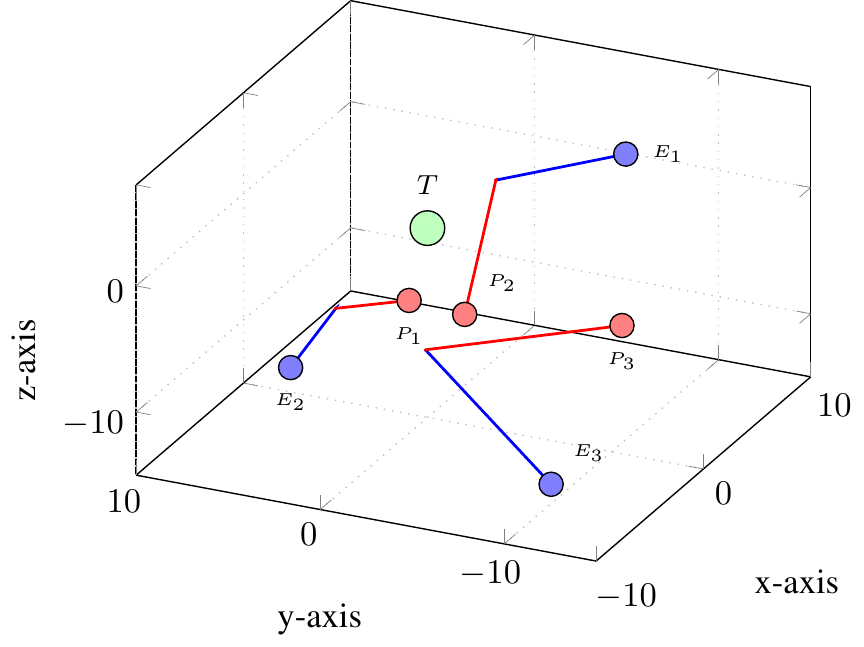}		
			\caption{Optimal Play}
			\label{fig:optimal}
		\end{subfigure}  
		\begin{subfigure}[h]{0.24\textwidth}
			\centering 
			\includegraphics[scale=.5]{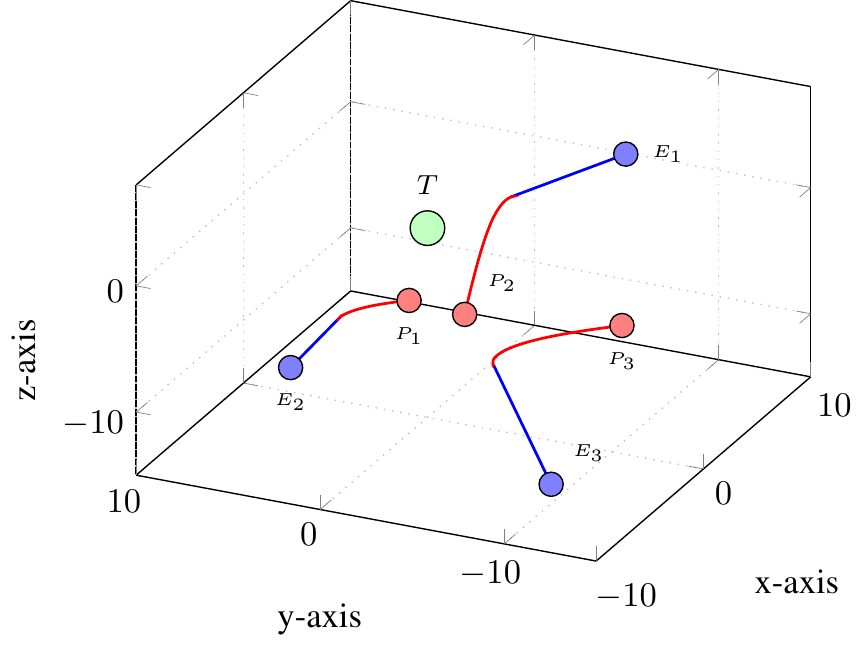}
			\caption{Nonoptimal Play}
			\label{fig:nonoptimal}
		\end{subfigure}
		\caption{Trajectories of a 3v3 game in $\mR_P$}
		\label{fig:num_ill}
	\end{figure}
	
	\begin{example}
		\label{ex:Re}
		Consider an MRADG with 3 pursuers and 3 evaders. The pursuers are situated at  $\x_{P_1}=(0.38,-7.06,1.17),\ \x_{P_2}=$($0.10,-7.45,-10.68)$ and $\x_{P_3}=(0.80,3.98,-8.45)$, while the evaders are situated at $\x_{E_1}=(-1.57,-6.23,1.67),\ \x_{E_2}=(0.38,-11.65,2.24)$ and $\x_{E_3}=(4.79,-4.71,2.68)$. The speeds of the pursuers are $V_{P_1}=2.09,\ V_{P_2}=1.65$ and $V_{P_3}=1.69$, while those of the evaders are $U_{E_1}=1.41,\ U_{E_2}=1.75$ and $U_{E_3}=1.83$. The evading team wins in this particular case and the optimization problem \eqref{eq:lpsolution} has two solutions $\bm\gamma_1^\star$ and $\bm\gamma_2^\star$ characterized by $\{21,12,33\}$ and $\{21,13,32\}$ respectively, as discussed in Example \ref{ex:refine}. The optimal trajectories as a result of both of these assignments are shown in Figure \ref{fig:evader_win}. The refined set $\Theta^\star$ contains a single element $\Th$ characterized by $\{21,12,33\}$ which equals $\bm\gamma_1^\star$. Further, in this scenario, there is a possible pairing of pursuer $P_3$ and evader $E_2$ with $\B_{23}\geq0$ and $\alpha_{23}>1$ for which case we do not have a Value function. However, the set $\{ij\in A_\Th:\alpha_{ij}>1\}$ remains empty which allows us to provide a Value function for the complete game thus guaranteeing optimality of the strategies. 
	\end{example}
	
	\begin{figure} 
		\centering
		\begin{subfigure}[h]{0.24\textwidth}
			\centering
			\includegraphics[scale=.5]{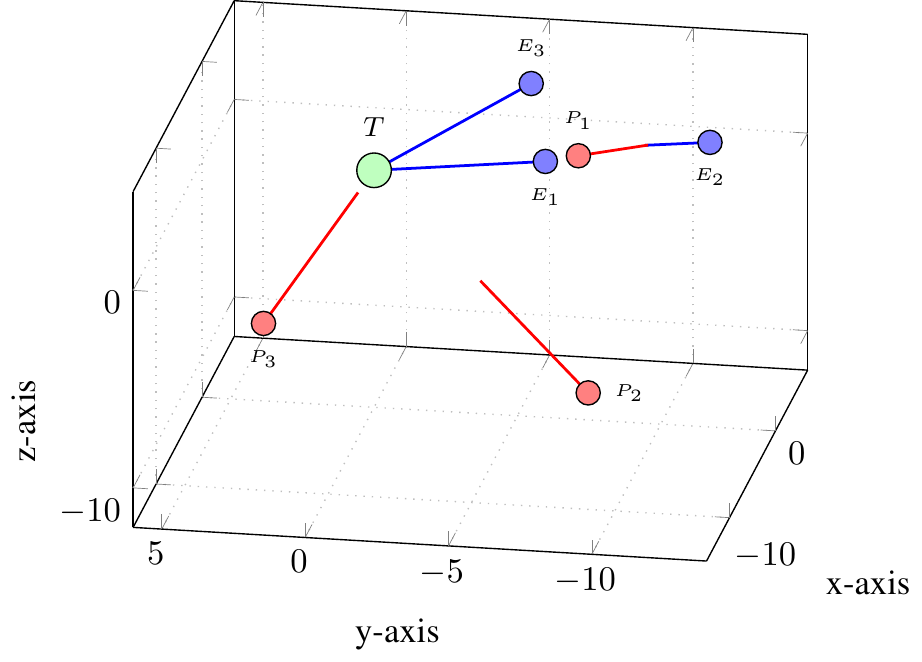}		
			\caption{First assignment}
			\label{fig:evader_win_1}
		\end{subfigure}  
		\begin{subfigure}[h]{0.24\textwidth}
			\centering 
			\includegraphics[scale=.5]{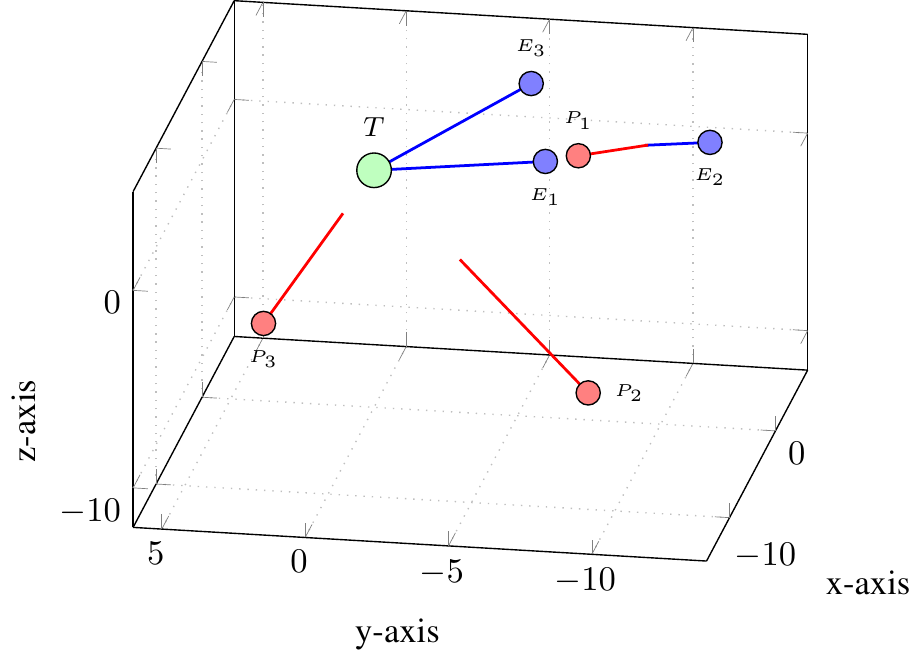}
			\caption{Second assignment}
			\label{fig:evader_win_2}
		\end{subfigure}
		\caption{Trajectories of a 3v3 game in $\mR_E$}
		\label{fig:evader_win}
	\end{figure}
	
	\begin{example}
		\label{ex:ds}
		Now, consider a reach-avoid game with $3$ pursuers and $2$ evaders with $\x_{P_1}=(1,0,0),\ \x_{P_2}=(1,0,0.5),\ \x_{P_3}=(1,0,-0.5),\ \x_{E_1}=(0.75,1,0)$ and $\x_{E_2}=(0.75,-1,0)$. All the evaders are assumed to have a speed of 0.5 units while the pursuers have unit speed. Amongst the $6$ feasible assignments, $4$ achieve the optimal Value and hence the state belongs to a dispersal surface. Thus at the beginning of the game, there are $4$ equally optimal solutions for the players given by the assignments $\{13,21\}$, $\{12,21\}$, $\{11,23\}$ and $\{11,22\}$ resulting in a Value of $1.54$. Hence, this state lies on the dispersal surface discussed in the proof of Theorem \ref{thm:valueP}. The optimal trajectories in each of these cases are shown in Figure \ref{fig:dispersal}.
	\end{example} 
	\begin{figure}  
		\begin{subfigure}[h]{.175\textwidth} 
			\includegraphics[scale=0.5]{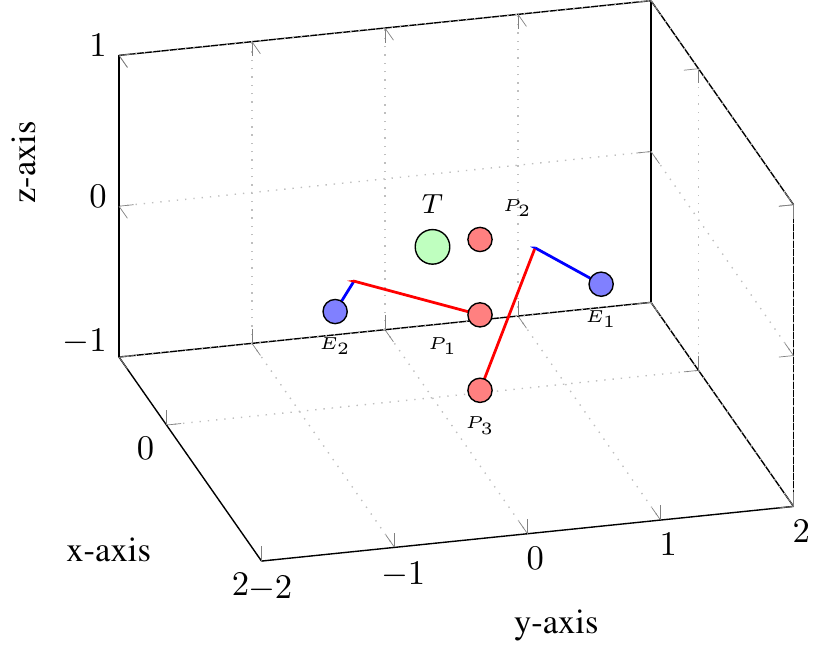}
			\label{fig:dispersal1}
		\end{subfigure} \hspace{20pt}
		\begin{subfigure}[h]{.175\textwidth} 
			\includegraphics[scale=0.5]{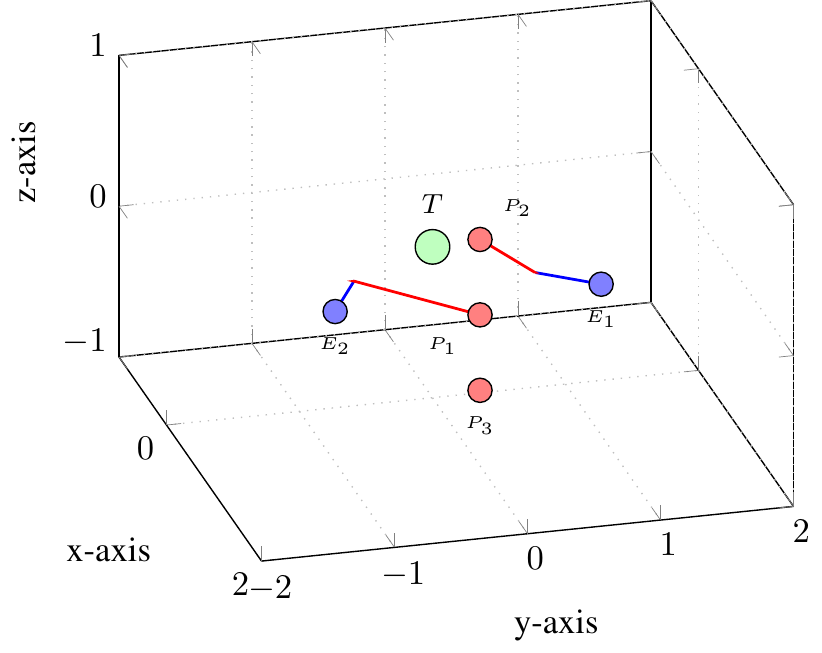}
			\label{fig:dispersal2}
		\end{subfigure} \\
		\begin{subfigure}[h]{.175\textwidth} 
			\includegraphics[scale=0.5]{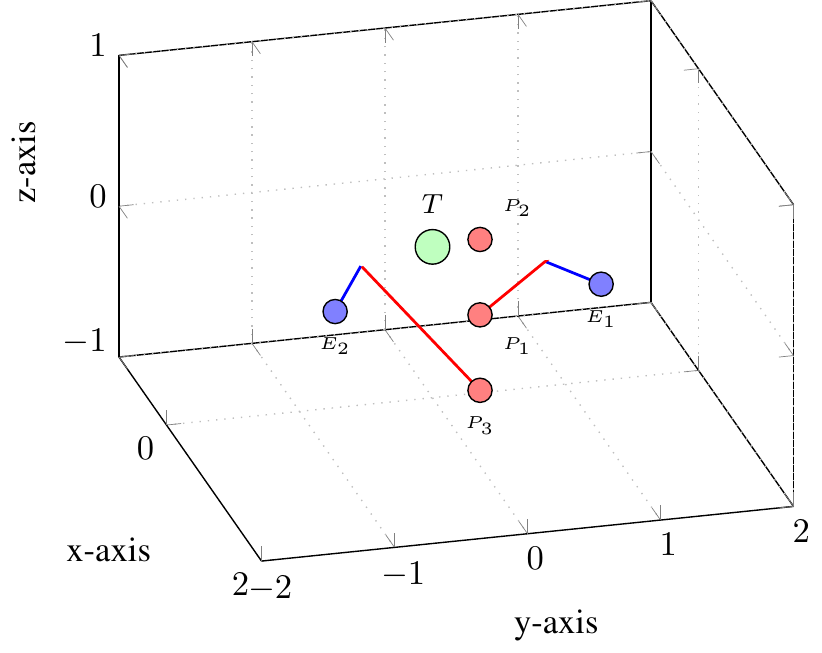}
			\label{fig:dispersal3}
		\end{subfigure} \hspace{20pt}
		\begin{subfigure}[h]{.175\textwidth} 
			\includegraphics[scale=0.5]{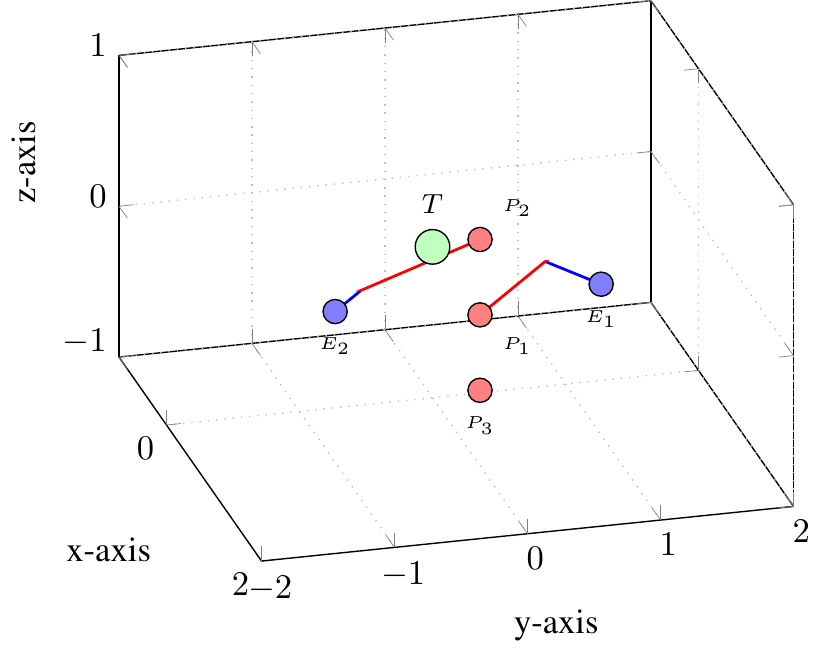}
			\label{fig:dispersal4}
		\end{subfigure} 
		\caption{Optimal assignments for dispersal surface}
		\label{fig:dispersal} 
	\end{figure}
	
	\section{Conclusion and Future Work}
	\label{sec:conc}
	This paper studies an MRADG in 3D space with $n$ pursuers and $m$ evaders, where $n\geq m \geq 1$. We propose an optimal task assignment algorithm based on linear programming to provide optimal trajectories for all players while satisfying the HJI-PDE. Currently, our work makes three critical assumptions (see Assumption \ref{assum:assumption1}) to solve the MRADG. This work can be further extended by relaxing each one of these assumptions. First, the assignment scheme may allow assigning multiple pursuers to one evader. This further restricts the reachable space of the evader and may potentially increase the Value of the game. Next, a multiplayer scenario can be considered with $n<m$. To address such a setting, a mechanism allowing a pursuer to capture multiple evaders must be developed. Finally, the assumption of commitment to the initial assignment may be relaxed. This will allow the pursuer team to dynamically compute the assignment scheme based on the current state making the strategy more robust and able to handle deviations from the optimal strategies. These are some of the directions that can be explored in future work. 
	
	\bibliographystyle{IEEEtran}
	\bibliography{mradg.bib}

\end{document}